\documentclass{article}
\usepackage{amsmath,amssymb,amsthm,amsxtra,bm,graphicx}
\usepackage[all]{xy}
\usepackage[hypertex]{hyperref}

\newtheorem{thm}{Theorem}[section]
\newtheorem{prop}[thm]{Proposition}
\newtheorem{lem}[thm]{Lemma}

\newtheorem{conj}[thm]{Conjecture}
\newtheorem*{conjA}{Conjecture~A}
\newtheorem*{mthm}{Main Theorem}
\theoremstyle{definition}
\newtheorem{dfn}[thm]{Definition}
\newtheorem{rem}[thm]{Remark}
\newtheorem{notation}[thm]{Notation}
\newtheorem{assumption}[thm]{Assumption}
\newtheorem*{example}{Example}
\newtheorem{construction}[thm]{Construction}

\newcommand{\N}{\mathbb{N}}

\newcommand{\Z}{\mathbb{Z}}

\newcommand{\bvec}[1]{\mbox{\boldmath $#1$}}
\newcommand{\ab}[2]{|#1|_0(#2)}

\newcommand{\nab}[2]{|#1|^{\mathrm{naive}}_0(#2)}
\newcommand{\RO}[1]{O((\log{(1/\rho)})^{#1})}
\newcommand{\Ro}[1]{(\log{(1/\rho)})^{#1}}

\newcommand{\bk}[1]{i(#1)}
\newcommand{\s}[1]{s_{#1}}
\newcommand{\m}[1]{m(#1)}
\newcommand{\NN}[1]{N(#1)}
\newcommand{\x}[1]{v(#1)}
\newcommand{\xx}{v}

\newcommand{\alg}{\mathrm{alg}}
\newcommand{\naive}{\mathrm{naive}}
\newcommand{\ur}{\mathrm{ur}}
\newcommand{\Hom}{\mathrm{Hom}}
\newcommand{\NP}{\mathrm{NP}}
\newcommand{\Sol}{\mathrm{Sol}}
\newcommand{\con}{\mathrm{con}}

\newcommand{\an}{\mathrm{an}}
\newcommand{\Fil}{\mathrm{Fil}}

\begin{document}

\title{On the rationality and continuity of logarithmic growth filtration of solutions of $p$-adic differential equations}
\author{Shun Ohkubo
\footnote{
Graduate School of Mathematics, Nagoya University, Furocho, Chikusaku, Nagoya 464-8602, Japan. E-mail address: shun.ohkubo@gmail.com}
}
\date{\today}

\maketitle

\begin{abstract}
We study the asymptotic behavior of solutions of Frobenius equations defined over the ring of overconvergent series. As an application, we prove Chiarellotto-Tsuzuki's conjecture on the rationality and right continuity of Dwork's logarithmic growth filtrations associated to ordinary linear $p$-adic differential equations with Frobenius structures.
\end{abstract}

\tableofcontents

\section{Introduction}

We consider an ordinary linear $p$-adic differential equation
\[
Dy=\frac{d^ny}{dx^n}+a_{n-1}\frac{d^{n-1}y}{dx^{n-1}}+\dots+a_0y=0,
\]
whose coefficients are bounded on the $p$-adic open unit disc $|x|<1$. We define its solution space by
\[
\Sol(D):=\{y\in\mathbb{Q}_p[\![x]\!];Dy=0\}.
\]
In her study of $p$-adic elliptic functions, Lutz proves that any solution $y$ of $Dy=0$ has a non-zero radius of convergence $r$ (\cite[Th\'eor\`eme~IV]{Lut}). In the paper \cite{Dwo}, Dwork studies the asymptotic behavior of $y$ near the boundary $|x|=r$ assuming that any solution of $Dy=0$ converges in a common open disc $|x|<r$. For simplicity, we assume $r=1$. The most general result in this viewpoint is that $y$ has a logarithmic growth (log-growth) $n-1$, that is,
\[
\sup_{|x|=\rho}|y(x)|=\RO{1-n}\text{ as }\rho\uparrow 1.
\]
Dwork also defines the so-called special log-growth filtration of $\Sol(D)$ by
\[
\Sol_{\lambda}(D):=\{y\in\Sol(D);\sup_{|x|=\rho}|y(x)|=\RO{-\lambda}\text{ as }\rho\uparrow 1 \}.
\]

We assume that the $a_i$'s are rational functions over $\mathbb{Q}_p$. Over the $p$-adic field, a na\"ive analogue of analytic continuation fails. In particular, the existence of local solutions of $Dy=0$ at the disc $|x-a|<1$ for any $a$ does not imply the existence of global solutions. Even the $p$-adic exponential series $e^x=1+x+2^{-1}x^2+\dots$, which is a solution of $dy/dx=y$, has a radius of convergence $p^{-1/(p-1)}$. Hence, it is natural to ask how the special log-growth filtration varies from disc to disc. Assume $p\neq 2$. In \cite{book}, Dwork provides an answer to this question for the hypergeometric differential equation
\[
Dy=x(1-x)\frac{d^2y}{dx^2}+(1-2x)\frac{dy}{dx}-\frac{1}{4}y=0,
\]
which arises from the Legendre family of elliptic curves over $\mathbb{F}_p$
\[
E_x:z^2=w(w-1)(w-x),\ x\neq 0,1.
\]
Owing to its geometric origin, the hypergeometric differential equation admits a Frobenius structure: let $\bar{a}\in\mathbb{F}_p\setminus\{0,1\}$, and let $a\in\mathbb{Z}_p$ be a lift of $\bar{a}$. The Frobenius slopes of the solution space of $Dy=0$ at the disc $|x-a|<1$ are $0,1$ if $E_{\bar{a}}$ is ordinary, and $1/2,1/2$ if $E_{\bar{a}}$ is supersingular. Dwork proves that the special log-growth filtration at the disc $|x-a|<1$ coincides with the Frobenius slope filtration at the disc $|x-a|<1$.

In the last few decades, $p$-adic differential equations have been extensively studied from many perspectives. As for the existence of solutions, Andr\'e, Kedlaya, and Mebkhout (\cite{And0},\cite{Ann},\cite{Meb}) independently prove the $p$-adic local monodromy theorem, which asserts the quasi-unipotence of $p$-adic differential equations defined over the Robba ring with Frobenius structures. Additionally, several striking applications of $p$-adic differential equations emerge: for example, Berger relates a certain $p$-adic representation of the absolute Galois group of $\mathbb{Q}_p$ to a $p$-adic differential equation over the Robba ring; then, he proves Fontaine's $p$-adic monodromy conjecture by using the $p$-adic local monodromy theorem (\cite{Ber}).

However, Dwork's works on the log-growth of solutions of $p$-adic differential equations have been neglected for a long period until Chiarellotto and Tsuzuki drew attention to it in \cite{CT}. We briefly summarize some recent developments on this subject.

\begin{enumerate}
\item[$\bullet$] In \cite{CT}, Chiarellotto and Tsuzuki formulate a fundamental conjecture on the log-growth filtrations for $p$-adic differential equations with Frobenius structures (see Conjecture~\ref{conj:sp}). Their conjecture is two-fold. The first part can be stated as follows:
\begin{conjA}[{Conjecture~\ref{conj:sp}~(i)}]
Let $Dy=0$ be a $p$-adic differential equation with a Frobenius structure. Then, the breaks of the filtration $\Sol_{\bullet}(D)$ are rational and $\Sol_{\lambda}(D)=\cap_{\mu>\lambda}\Sol_{\mu}(D)$ for all $\lambda\in\mathbb{R}$.
\end{conjA}
The second part is about a comparison of the log-growth filtration and the Frobenius slope filtration under a certain technical assumption, which is based on Dwork's work on the hypergeometric differential equation. They prove the conjecture in the rank~$2$ case in \cite{CT}. They also provide a complete answer to a generic version of their conjecture in \cite{CT2}.

\item[$\bullet$] In \cite{And}, Andr\'e proves Dwork's conjecture on a specialization property for the log-growth filtration, which is an analogue of Grothendieck-Katz specialization theorem on Frobenius structure.

\item[$\bullet$] In \cite{pde}, Kedlaya studies effective convergence bounds on the solutions of $p$-adic differential equations with nilpotent singularities, which allows the $a_i$'s to have a pole at $x=0$. Then, he proves a partial generalization of Chiarellotto-Tsuzuki's earlier works to $p$-adic differential equations with nilpotent singularities.
\end{enumerate}

Our main result in this paper is
\begin{mthm}[{Theorem~\ref{thm:main}~(i)}]
Conjecture~A is true.
\end{mthm}
Under a certain technical assumption, we also prove the second part of Chiarellotto-Tsuzuki's conjecture (Theorem~\ref{thm:main}~(ii)).

\subsection*{Strategy of proof}

We present the proof of the rationality of breaks of the filtration $\Sol_{\bullet}(D)$. Let $\mathbb{Q}_p[\![x]\!]_0:=\mathbb{Z}_p[\![x]\!][p^{-1}]$ be the ring of bounded functions on the open unit disc, and $\sigma$ a $\mathbb{Q}_p$-algebra endomorphism of $\mathbb{Q}_p[\![x]\!]_0$ such that $\sigma(x)=x^p$. Instead of a na\"ive $p$-adic differential equation $Dy=0$, we consider a finite free $\mathbb{Q}_p[\![x]\!]_0$-module $M$ of rank $n$ endowed with an action of $d/dx$. The existence of a Frobenius structure of $Dy=0$ is equivalent to the existence of a $\sigma$-semi-linear structure $\varphi$ on $M$ compatible with $\nabla$. In \cite{CT}, Chiarellotto and Tsuzuki establish a standard method for studying the log-growth filtration associated to $M$ as follows. We fix a cyclic vector $e$ of $M$ as a $\sigma$-module over the fraction field of $\mathbb{Q}_p[\![x]\!]_0$. Let $V(M)$ be the set of horizontal sections of $M$ after tensoring with the ring of analytic functions over the open unit disc. Let $v\in V(M)$ be a Frobenius eigenvector, i.e., $\varphi(v)=\lambda v$ for some $\lambda\in\mathbb{Q}_p$. If we write $v$ as a linear combination of $e,\varphi(e),\dots,\varphi^{n-1}(e)$, then the coefficient $f$ of $\varphi^{n-1}(e)$ satisfies a certain Frobenius equation
\[
b_nf^{\sigma^n}+b_{n-1}f^{\sigma^{n-1}}+\dots+b_0f=0,\ b_i\in \mathbb{Q}_p[\![x]\!]_0.
\]
Then, the rationality of breaks of $\Sol_{\bullet}(D)$ is reduced to the rationality of the log-growth of $f$, i.e., the existence of $\lambda\in\mathbb{Q}$ such that
\[
\sup_{|x|=\rho}|f(x)|=\RO{-\lambda}\text{ as }\rho\uparrow 1
\]
and
\[
\sup_{|x|=\rho}|f(x)|\neq \RO{-\mu}\text{ as }\rho\uparrow 1
\]
for any $\mu<\lambda$. The rationality of the log-growth of $f$ is proved by Chiarellotto and Tsuzuki in \cite{CT} when $n=2$, then by Nakagawa in \cite{Nak} when $n$ is arbitrary under the assumption that the number of breaks of the Newton polygon of $b_nX^n+b_{n-1}X^{n-1}+\dots+b_0$ as a polynomial over the Amice ring $\mathcal{E}$ is equal to $n$. Nakagawa's assumption is too strong since it is equivalent to assuming that the number of breaks of the Frobenius filtration of $M$ tensored with $\mathcal{E}$ is equal to $n$. Unfortunately, a na\"ive attempt to generalize Nakagawa's result without the assumption on the Newton polygon seems to fail.

To overcome this difficulty, we carefully choose a cyclic vector $e$ in \S~\ref{sec:gen}: by definition, the Newton polygon of $b_nX^n+b_{n-1}X^{n-1}+\dots+b_0$ is the boundary of the lower convex hull of some set of points associated to the $b_i$'s. Our requirement for $e$ is that each plotted point belongs to the Newton polygon. The construction of $e$ is performed after a certain base change which is described in Kedlaya's framework of analytic rings. By using our cyclic vector $e$, the corresponding Frobenius equation is defined over Kedlaya's ring. Hence, we need to introduce a notion of log-growth on Kedlaya's ring (\S~\ref{sec:lg}). Then, we generalize Nakagawa's calculation in \S~\ref{sec:Frob}. Finally, we obtain the rationality of the log-growth filtration of $V(M)$ in \S~\ref{sec:pf}.

\section{Summary of notation}

We summarize our notation in this paper. Basically, we adopt the notation in \cite{CT2}. In the appendix, we have a diagram describing relations between various rings defined in the following.

\subsection{Coefficient rings}
\begin{enumerate}
\item [$p$]: a prime number.

\item[$K$]: a complete discrete valuation field of characteristic $(0,p)$.

\item[$\mathcal{O}_K$]: the integer ring of $K$.

\item[$k_K$]: the residue field of $K$.

\item[$\pi_K$]: a uniformizer of $\mathcal{O}_K$.

\item[$|\cdot|$]: the $p$-adic absolute value on $K^{\alg}$ associated to a valuation of $K$, normalized by $|p|=p^{-1}$.

\item[$q$]: a positive power of $p$.

\item[$q^s\in\mathbb{Q}$]: Let $s$ be a rational number and write $s=a/b$ with relatively prime $a,b\in\mathbb{Z}$. The notation ``$q^s\in\mathbb{Q}$'' means that $b$ divides $\log_p{q}$, and we put $q^s:=p^{a\log_p{q}/b}$.

\item[$\sigma$]: a $q$-Frobenius on $\mathcal{O}_K$, i.e., a local ring endomorphism of $\mathcal{O}_K$ such that $\sigma(a)\equiv a^q\mod{\pi_K\mathcal{O}_K}$.

\item[$K^{\sigma}$]: the inductive limit of $K$
\[\xymatrix{
K\ar[r]^{\sigma}&K\ar[r]^{\sigma}&\dots.
}\]
We regard $K^{\sigma}$ as an extension of $K$. Then, $K^{\sigma}$ is a Henselian discrete valuation field, whose value group coincides with the value group of $K$, with residue field $k_K^{p^{-\infty}}$.

\item[$K^{\sigma,\ur}$]: the completion of the maximal unramified extension of $K^{\sigma}$. Then, $K^{\sigma,\ur}$ is a complete discrete valuation field, whose value group coincides with the value group of $K$, with the residue field $k_K^{\alg}$. Moreover, $\sigma$ induces a $q$-Frobenius on $K^{\sigma,\ur}$.
\end{enumerate}

\subsection{Various rings of functions}
In the appendix, we have a diagram of the rings mentioned in this paper including the following rings of functions.

\begin{enumerate}
\item[$x$]: an indeterminate.

\item[$\nab{\cdot}{\rho}$]: the multiplicative map
\[
K[\![x]\!]\to \mathbb{R}_{\ge 0}\cup\{\infty\};\sum_{n\in\mathbb{N}}a_nx^n\mapsto\sup_{n\in\mathbb{N}}|a_n|\rho^n
\]
defined for $\rho\in [0,1]$.

\item[$K\{x\}$]: the $K$-algebra of analytic functions on the open unit disc $|x|<1$, i.e.,
\[
K\{x\}:=\left\{\sum_{n\in\mathbb{N}}a_nx^n\in K[\![x]\!];|a_n|\rho^n\to 0\ (n\to\infty)\forall \rho\in [0,1)\right\}.
\]
Note that $\nab{\cdot}{\rho}$ defines a multiplicative non-archimedean norm on $K\{x\}$ if $\rho\neq 0$.

\item[{$K[\![x]\!]_{\lambda}$}]: the Banach $K$-subspace of power series of logarithmic growth (log-growth) $\lambda$ in $K\{x\}$ for $\lambda\in \mathbb{R}_{\ge 0}$, i.e.,
\begin{align*}
K[\![x]\!]_{\lambda}:&=\left\{\sum_{n\in\mathbb{N}}a_nx^n\in K[\![x]\!];\sup_{n\in\mathbb{N}}|a_n|/(n+1)^{\lambda}<\infty\right\}\\
&=\left\{f\in K\{x\};\nab{f}{\rho}=\RO{-\lambda}\text{ as }\rho\uparrow 1\right\},
\end{align*}
where the last equality follows from \cite[Lemma~2.2.1~(iv)]{And}. Note that $K[\![x]\!]_0$ coincides with the ring of bounded functions on the open unit disc $|x|<1$, i.e.,
\[
K[\![x]\!]_0=\mathcal{O}_K[\![x]\!][\pi_K^{-1}]=\left\{\sum_{n\in\mathbb{N}}a_nx^n\in K[\![x]\!];\sup_{n\in\mathbb{N}}|a_n|<\infty\right\}.
\]
We define $K[\![x]\!]_{\lambda}:=0$ for $\lambda\in\mathbb{R}_{<0}$. Note that $K[\![x]\!]_{\lambda}$ is stable under the derivation $d/dx$.

\item[$\mathcal{E}$]: the fraction field of the $p$-adic completion of $\mathcal{O}_K[\![x]\!][x^{-1}]$, i.e.,

\[
\mathcal{E}:=\left\{\sum_{n\in\mathbb{Z}}a_nx^n\in K[\![x,x^{-1}]\!];\sup_{n\in\mathbb{Z}}|a_n|<\infty, |a_n|\to 0\text{ as }n\to-\infty\right\}.
\]
Note that $\mathcal{E}$ is canonically endowed with a norm which is an extension of $\nab{\cdot}{1}$. Then, $(\mathcal{E},\nab{\cdot}{1})$ is a complete discrete valuation field of mixed characteristic $(0,p)$ with uniformizer $\pi_K$ and residue field $k_K(\!(x)\!)$.

\item[$\mathcal{E}^{\dagger}$]: the ring of overconvergent power series in $\mathcal{E}$, i.e.,
\[
\mathcal{E}^{\dagger}:=\left\{\sum_{n\in\mathbb{Z}}a_nx^n\in \mathcal{E};|a_n|\rho^n\to 0\ (n\to-\infty)\text{ for some }\rho\in (0,1)\right\}.
\]
Note that $(\mathcal{E}^{\dagger},\nab{\cdot}{1})$ is a Henselian discrete valuation field whose completion is $\mathcal{E}$.

\item[$\mathcal{R}$]: the Robba ring with variable $x$ and coefficient $K$, i.e.,
\[
\mathcal{R}:=\left\{\sum_{n\in\mathbb{Z}}a_nx^n\in K[\![x,x^{-1}]\!];|a_n|\rho^n\to 0\ (n\to\pm \infty)\ \forall \rho\in (\rho_0,1)\text{ for some }\rho_0\in (0,1)\right\}.
\]

\item[$\sigma$]: a $q$-Frobenius on $\mathcal{O}_K[\![x]\!]$, which is an extension of $\sigma$, defined by fixing $\sigma(x)= x^q\mod{\mathfrak{m}_K\mathcal{O}_K[\![x]\!]}$. Note that $\sigma$ induces ring endomorphisms on $K[\![x]\!]$, $K\{x\}$, $\mathcal{E}^{\dagger}$, $\mathcal{E}$, and $\mathcal{R}$, and $K[\![x]\!]_{\lambda}$ is stable under $\sigma$ (\cite[4.6.4]{Chr}).

\item[$\mathcal{E}_t$]: a copy of $\mathcal{E}$ in which $x$ is replaced by another indeterminate $t$. As above, we regard $\mathcal{E}_t$ as a complete discrete valuation field where $t$ is a $p$-adic unit. In the literature, $t$ is called Dwork's generic point (\cite[Definition~9.7.1]{pde}).

\item[{$\mathcal{E}_t[\![X-t]\!]_0$}]: the ring of bounded functions on $|X-t|<1$ with variable $X-t$ and coefficient $\mathcal{E}_t$. We endow $\mathcal{E}_t[\![X-t]\!]_0$ with $\mathcal{E}$-algebra structure by the $K$-algebra homomorphism
\[
\tau:\mathcal{E}\to\mathcal{E}_t[\![X-t]\!]_0;f\mapsto\sum_{n\in\mathbb{N}}\frac{1}{n!}\left.\left(\frac{d^nf}{dx^n}\right)\right|_{x=t}(X-t)^n.
\]
Since $\tau(K)\subset \mathcal{E}_t$ and $\tau(x)=X$, $\tau$ is equivariant under the derivations $d/dx$ and $d/dX$. We define a $q$-Frobenius on $\mathcal{E}_t[\![X-t]\!]_0$ by $\sigma|_{\mathcal{E}_t}=\sigma$ (by identifying $t$ as $x$) and $\sigma(X-t)=\tau(\sigma(x))-\sigma(x)|_{x=t}$. Then, $\tau$ is also $\sigma$-equivariant.

\item[{$\mathcal{E}_t[\![X-t]\!]_{\lambda}$}]: the Banach $\mathcal{E}_t$-subspace of power series of log-growth $\lambda$ in $\mathcal{E}_t\{X-t\}$.
\end{enumerate}
Let $R$ be either $K[\![x]\!]_0$, $K\{x\}$, $\mathcal{E}^{\dagger}$, $\mathcal{E}$, or $\mathcal{R}$. We define $\Omega^1_R:=Rdx$ with a $K$-linear derivation $d:R\to \Omega^1_R;f\mapsto (df/dx)dx$. We also endow $\Omega^1_R$ with a semi-linear $\sigma$-action defined by $\sigma(dx):=d\sigma(x)$. For $\mathcal{E}_t[\![X-t]\!]_0$ and $\mathcal{E}_t\{X-t\}$, we also define a corresponding $\Omega^1_{\bullet}$ by replacing $K$ and $x$ by $\mathcal{E}_t$ and $X-t$, respectively.

\subsection{Filtration and Newton polygon}

Let $V$ be a finite dimensional vector space over a field $F$. Let $V^{\bullet}=\{V^{\lambda}\}_{\lambda\in\mathbb{R}}$ be a decreasing filtration by subspaces of $V$. Then, we define
\[
V^{\lambda-}:=\bigcap_{\mu<\lambda}V^{\mu},\ V^{\lambda+}:=\bigcup_{\mu>\lambda}V^{\mu}.
\]
We say that $\lambda\in\mathbb{R}$ is a break of $V^{\bullet}$ if $V^{\lambda-}\neq V^{\lambda+}$. We also define the multiplicity of $\lambda$ as $\dim_FV^{\lambda-}-\dim_FV^{\lambda+}$. We say that $V^{\bullet}$ is rational if all breaks of $V^{\bullet}$ are rational. We say that $V^{\bullet}$ is right continuous if $V^{\lambda}=V^{\lambda+}$ for all $\lambda\in\mathbb{R}$. We say that $V^{\bullet}$ is exhaustive or separated if $\cup_{\lambda\in\mathbb{R}}V^{\lambda}=V$ or $\cap_{\lambda\in\mathbb{R}}V^{\lambda}=0$, respectively.

Similarly, for an increasing filtration $V_{\bullet}=\{V_{\lambda}\}_{\lambda\in\mathbb{R}}$ of $V$, we define
\[
V_{\lambda-}:=\bigcup_{\mu<\lambda}V_{\mu},\ V_{\lambda+}:=\bigcap_{\mu>\lambda}V_{\mu}.
\]
We also define a break, a rationality, and a right continuity of $V_{\bullet}$ by replacing superscripts by subscripts.

We define the Newton polygon of a filtration as follows (\cite[3.3]{CT}). Let $\{V^{\lambda}\}_{\lambda\in\mathbb{R}}$ (resp. $\{V_{\lambda}\}_{\lambda\in\mathbb{R}}$) be a decreasing (resp. increasing) filtration of $V$. Let $\lambda_1<\dots<\lambda_n$ be the breaks of the filtration $V^{\bullet}$ (resp. $V_{\bullet}$) with multiplicities $m_1,\dots,m_n$. We define the Newton polygon of $V^{\bullet}$ (resp. $V_{\bullet}$) as the piecewise linear function in the $xy$-plane whose left endpoint is $(0,0)$, with slopes $\lambda_1,\dots,\lambda_n$ whose projections to the $x$-axis have lengths $m_1,\dots,m_n$.

\section{Chiarellotto-Tsuzuki's conjectures and main theorem}

We first recall the definition of $(\sigma,\nabla)$-modules over $K[\![x]\!]_0$ and $\mathcal{E}$. Then, we recall the definition of the log-growth filtrations for $(\sigma,\nabla)$-modules over $K[\![x]\!]_0$ and $\mathcal{E}$, and recall Chiarellotto-Tsuzuki's conjectures. After recalling known results on the conjectures, we state our main results. Our basic references are \cite{CT}, \cite{CT2}, and \cite{pde}.

\subsection{$\sigma$-modules}\label{subsec:phi}

Let $R$ be a commutative ring with a ring endomorphism $\delta$. We denote $\delta(r)$ by $r^{\delta}$ if no confusion arises. A $\delta$-module $M$ is a finite free $R$-module $M$ endowed with an $R$-linear isomorphism $\varphi:\delta^*M:=R\otimes_{\delta,R}M\to M$. We can view $M$ as a left module over the twisted polynomial ring $R\{\delta\}$ (\cite[14.2.1]{pde}). If we regard $\varphi$ as a $\delta$-linear endomorphism of $M$, then $(M,\varphi^n)$ for $n\in\mathbb{N}$ is a $\delta^n$-module over $R$. For $\alpha\in R^{\times}$, $(M,\alpha\varphi)$ is also a $\delta$-module over $R$.

Let $M$ be a $\sigma$-module over $K$ ($K$ might be $\mathcal{E}$). We recall the Frobenius slope filtration of $M$ (\cite[\S~2]{CT}). We say that $M$ is \'etale if there exists an $\mathcal{O}_K$-lattice $\mathfrak{M}$ of $M$ such that $\varphi(\mathfrak{M})\subset \mathfrak{M}$ and $\varphi(\mathfrak{M})$ generates $\mathfrak{M}$. We say that $M$ is pure of slope $\lambda\in\mathbb{R}$ if there exists $n\in\mathbb{N}_{>0}$ and $\alpha\in K$ such that $\log_{q^n}|\alpha|=-\lambda$ and $(M,\alpha^{-1}\varphi^n)$ is \'etale (\cite[2.1]{CT}). For a $\sigma$-module $M$ over $K$, there exists a unique increasing filtration $\{S_{\lambda}(M)\}_{\lambda\in\mathbb{R}}$, called the slope filtration,  of $M$ such that $S_{\lambda}(M)/S_{\lambda-}(M)$ is pure of slope $\lambda$. We call the breaks of $S_{\bullet}(M)$ the Frobenius slopes of $M$. The following are some basic properties of the slope filtration:
\begin{itemize}
\item The slope filtration of $M$ is exhaustive, separated, and right continuous.

\item The Frobenius slopes of $M$ are rational.

\item The slope filtration of $(M,\varphi^n)$ is independent of the choice of $n\in\mathbb{N}_{>0}$.
\end{itemize}
Assume that $k_K$ is algebraically closed. Then, any short exact sequence of $\sigma$-modules splits (\cite[14.3.4, 14.6.6]{pde}). Moreover, let $M$ be a $\sigma$-module over $K$ such that $q^{\lambda}\in\mathbb{Q}$ for any Frobenius slope $\lambda$ of $M$. Then, $M$ admits a basis consisting of elements of the form $\varphi(v)=q^{\lambda} v$ (\cite[14.6.4]{pde}); we call $v$ a Frobenius eigenvector of slope $\lambda$. In this situation, for any $\sigma$-submodules $M'$ and $M''$ of $M$, we have $M'\subset M''$ if and only if any Frobenius eigenvector $v$ of $M'$ belongs to $M''$.

\subsection{Log-growth filtration}\label{subsec:lgfil}

Let $R$ be either $K[\![x]\!]_0$ ($K$ might be $\mathcal{E}$), $\mathcal{E}^{\dagger}$, $\mathcal{E}$, or $\mathcal{R}$. A $\nabla$-module over $R$ is a finite free $R$-module $M$ endowed with a connection, i.e., a $K$-linear map
\[
\nabla:M\to M\otimes_{R}\Omega^1_R=Mdx
\]
satisfying
\[
\nabla(am)=m\otimes da+a\nabla(m)
\]
for $a\in R$ and $m\in M$. A $(\sigma,\nabla)$-module over $R$ is a $\sigma$-module $(M,\varphi)$ over $R$ with a connection $\nabla$ such that the following diagram is commutative:
\[\xymatrix{
M\ar[r]^(.35){\nabla}\ar[d]^{\varphi}&M\otimes_R \Omega^1_R\ar[d]^{\varphi\otimes\sigma}\\
M\ar[r]^(.35){\nabla}&M\otimes_R \Omega^1_R.
}\]

\begin{enumerate}
\item[(I)] Special log-growth filtration (\cite[4.2]{CT})
\end{enumerate}

Let $M$ be a $(\sigma,\nabla)$-module of rank~$n$ over $K[\![x]\!]_0$. We define the space of horizontal sections of $M$ by
\[
V(M):=(M\otimes_{K[\![x]\!]_0}K\{x\})^{\nabla=0}
\]
and define the space of solutions of $M$ by
\begin{align*}
\Sol(M)&:=\Hom_{K[\![x]\!]_0}(M,K\{x\})^{\nabla=0}\\
&:=\{f\in \Hom_{K[\![x]\!]_0}(M,K\{x\});d(f(m))=(f\otimes \mathrm{id})(\nabla(m))\forall m\in M\}.\ (\text{see } \cite[p.~82]{pde})
\end{align*}
Both $V(M)$ and $\Sol(M)$ are known to be $K$-vector spaces of dimension $n$, and there exists a perfect pairing
\[
V(M)\otimes_K \Sol(M)\to K
\]
induced by the canonical pairing $M\otimes_{K[\![x]\!]_0}M\spcheck\to K[\![x]\!]_0$, where $M\spcheck$ denotes the dual of $M$. For $\lambda\in\mathbb{R}$, we define
\[
\Sol_{\lambda}(M):=\Hom_{K[\![x]\!]_0}(M,K[\![x]\!]_{\lambda})\cap \Sol(M),
\]
which induces an increasing filtration of $\Sol(M)$. We say that $M$ is solvable in $K[\![x]\!]_{\lambda}$ if $\dim_K\Sol_{\lambda}(M)=n$. We define
\[
V(M)^{\lambda}:=\Sol_{\lambda}(M)^{\perp},
\]
where $(\cdot)^{\perp}$ denotes the orthogonal space with respect to the above pairing. We call the decreasing filtration $\{V(M)^{\lambda}\}_{\lambda}$ the special log-growth filtration of $M$. Note that $\Sol_{\bullet}(M)$ and $V(M)^{\bullet}$ are exhaustive and separated. Moreover, $V(M)^{\lambda}$ (resp. $\Sol_{\lambda}(M)$) is a $\sigma$-submodule of $V(M)$ (resp. $\Sol(M)$) (\cite[4.8]{CT}).

\begin{example}\label{ex:sol}
Let
\[
Dy=\frac{d^ny}{dx^n}+a_{n-1}\frac{d^{n-1}y}{dx^{n-1}}+\dots+a_0y=0,\ a_i\in K[\![x]\!]_0
\]
be an ordinary linear $p$-adic differential equation. As in the introduction, we define
\[
\Sol(D):=\{y\in K[\![x]\!];Dy=0\}\supset\Sol_{\lambda}(D):=\{y\in K[\![x]\!]_{\lambda};Dy=0\}.
\]
We define a $\nabla$-module $M:=K[\![x]\!]_0e_0\oplus\dots\oplus K[\![x]\!]_0e_{n-1}$ by
\[
\nabla(e_i)=
\begin{cases}
e_{i+1}dx&\text{if }0\le i\le n-2\\
-(a_{n-1}e_{n-1}+\dots+a_0e_0)dx&\text{if }i=n-1.
\end{cases}
\]
Then, we have the canonical isomorphism
\[
\Sol(M)\to\Sol(D);f\mapsto f(e_0),
\]
under which we have
\[
\Sol_{\lambda}(M)=\Sol_{\lambda}(D).
\]
\end{example}

\begin{enumerate}
\item[(II)] Generic log-growth filtration (\cite[\S~4.1]{CT})
\end{enumerate}

Let $M$ be a $(\sigma,\nabla)$-module over $\mathcal{E}$. We denote by $\tau^*M$ the pull-back of $M$ under $\tau:\mathcal{E}\to\mathcal{E}_t[\![X-t]\!]_0$, which is a $(\sigma,\nabla)$-module over $\mathcal{E}_t[\![X-t]\!]_0$. By a theorem of Robba, there exists a unique $(\sigma,\nabla)$-submodule $M^{\lambda}$ of $M$ for $\lambda\in\mathbb{R}$ characterized as a minimal $(\sigma,\nabla)$-submodule of $M$ such that $\tau^*(M/M^{\lambda})$ is solvable in $\mathcal{E}_t[\![X-t]\!]_{\lambda}$ (\cite[4.1]{CT}). We call the decreasing filtration $\{M^{\lambda}\}_{\lambda\in\mathbb{R}}$ of $M$ the log-growth filtration of $M$. Note that $M^{\bullet}$ is exhaustive and separated, and if $M\neq 0$, then $M^{\lambda}\neq M$ for $\lambda\in\mathbb{R}_{\ge 0}$.

There exists a dual version of the log-growth filtration: for $\lambda\in\mathbb{R}$, we set $M_{\lambda}:=((M\spcheck)^{\lambda})^{\perp}$, where $(\cdot)^{\perp}$ denotes the orthogonal space with respect to the canonical pairing $M\otimes_{\mathcal{E}}M\spcheck\to\mathcal{E}$. Then, $M_{\lambda}$ is a maximal $(\sigma,\nabla)$-submodule of $M$ such that $\tau^*M_{\lambda}$ is solvable in $\mathcal{E}_t[\![X-t]\!]_{\lambda}$. Note that if $M\neq 0$, then $M_{\lambda}\neq 0$ for $\lambda\in\mathbb{R}_{\ge 0}$ by $(M\spcheck)^{\lambda}\neq M\spcheck$.

Note that the Frobenius slope filtration of $M$ is stable under the action of $\nabla$ (\cite[6.2]{CT}).

\begin{dfn}
Let $M$ be a $(\sigma,\nabla)$-module over $K[\![x]\!]_0$.
\begin{enumerate}
\item The Frobenius slope filtration $S_{\bullet}(V(M))$ of $V(M)$ is called the special Frobenius filtration of $M$ (\cite[6.7]{CT}). We call a Frobenius slope of $V(M)$ a special Frobenius slope of $M$.
\item We put $M_{\mathcal{E}}:=\mathcal{E}\otimes_{K[\![x]\!]_0}M$, which is a $(\sigma,\nabla)$-module over $\mathcal{E}$. The Frobenius slope filtration $S_{\bullet}(M_{\mathcal{E}})$ of $M_{\mathcal{E}}$ is called the generic Frobenius filtration of $M$ (\cite[6.1]{CT}). We call a Frobenius slope of $M_{\mathcal{E}}$ a generic Frobenius slope of $M$.
\end{enumerate}
\end{dfn}

\subsection{Chiarellotto-Tsuzuki's conjectures}\label{subsec:CT}

In \cite[Concluding Remarks]{Dwob}, Dwork observes that the log-growth and Frobenius slope filtrations can be compared. To formulate conjectures based on his observation, Chiarellotto and Tsuzuki introduce the following technical conditions:

\begin{dfn}
\begin{enumerate}
\item (\cite[6.1]{CT2}) Let $M$ be a $(\sigma,\nabla)$-module over $\mathcal{E}$. We say that $M$ is pure of bounded quotient (PBQ for short) if $M/M^0$ is pure as a $\sigma$-module.
\item (\cite[5.1]{CT2}) Let $M$ be a $(\sigma,\nabla)$-module over $K[\![x]\!]_0$. We say that $M$ is PBQ if $M_{\mathcal{E}}$ is PBQ. We say that $M$ is horizontal of bounded quotient (HBQ for short) if there exists a quotient $\overline{M}$ of $M$ as a $(\sigma,\nabla)$-module such that there exists a canonical isomorphism $\overline{M}_{\mathcal{E}}\cong M_{\mathcal{E}}/M_{\mathcal{E}}^0$. Finally, we say that $M$ is horizontally pure of bounded quotient (HPBQ for short) if $M$ is PBQ and HBQ.
\end{enumerate}
\end{dfn}

The following conjectures are first formulated by Chiarellotto and Tsuzuki in \cite[\S~6.4]{CT}. In this paper, we use the equivalent forms in \cite{CT2}.

\begin{conj}[{the conjecture ${\bf LGF}_{K[\![x]\!]_0}$ (\cite[2.5]{CT2})}]\label{conj:sp}
Let $M$ be a $(\sigma,\nabla)$-module over $K[\![x]\!]_0$.
\begin{enumerate}
\item The special log-growth filtration of $M$ is rational and right continuous.
\item Let $\lambda_{\max}$ be the highest Frobenius slope of $M_{\mathcal{E}}$. If $M$ is PBQ, then we have
\[
V(M)^{\lambda}=(S_{\lambda-\lambda_{\max}}(V(M\spcheck)))^{\perp}
\]
for all $\lambda\in\mathbb{R}$. Here, $(\cdot)^{\perp}$ denotes the orthogonal space with respect to the canonical pairing $V(M)\otimes_K V(M\spcheck)\to K$.
\end{enumerate}
\end{conj}

\begin{conj}[{the conjecture ${\bf LGF}_{\mathcal{E}}$ (\cite[2.4]{CT2})}]\label{conj:gen}
Let $M$ be a $(\sigma,\nabla)$-module over $\mathcal{E}$.
\begin{enumerate}
\item The log-growth filtration of $M$ is rational and right continuous.
\item Let $\lambda_{\max}$ be the highest Frobenius slope of $M$. If $M$ is PBQ, then we have
\[
M^{\lambda}=(S_{\lambda-\lambda_{\max}}(M\spcheck))^{\perp}
\]
for all $\lambda\in\mathbb{R}$. Here, $(\cdot)^{\perp}$ denotes the orthogonal space with respect to the canonical pairing $M\otimes_{\mathcal{E}} M\spcheck\to \mathcal{E}$.
\end{enumerate}
\end{conj}

To prove Chiarellotto-Tsuzuki's conjectures, we may assume that $k_K$ is algebraically closed as remarked in \cite[p.~42]{CT2}. In the following, we recall known results on Chiarellotto-Tsuzuki's conjectures.

\begin{thm}[{\cite[Theorem~7.1, 7.2]{CT2}}]
The conjecture ${\bf LGF}_{\mathcal{E}}$ is true.
\end{thm}

Hence, the remaining part of Chiarellotto-Tsuzuki's conjectures is the conjecture ${\bf LGF}_{K[\![x]\!]_0}$.

\begin{thm}\label{thm:CT}
Let $M$ be a $(\sigma,\nabla)$-module of rank~$n$ over $K[\![x]\!]_0$.
\begin{enumerate}
\item (\cite[Theorem~7.1~(2)]{CT}) The conjecture ${\bf LGF}_{K[\![x]\!]_0}$ is true if $n\le 2$.
\item (\cite[Theorem~8.7]{CT2}) The conjecture ${\bf LGF}_{K[\![x]\!]_0}$~(i) is true if $M$ is HBQ.
\item (\cite[Theorem~6.17]{CT}) For all $\lambda\in\mathbb{R}$, we have
\[
V(M)^{\lambda}\subset (S_{\lambda-\lambda_{\max}}(V(M\spcheck)))^{\perp}.
\]
\item (\cite[Theorem~6.5]{CT2}) The conjecture ${\bf LGF}_{K[\![x]\!]_0}$~(ii) is true if $M$ is HPBQ.
\item (\cite[Proposition~7.3]{CT2}) If the conjecture ${\bf LGF}_{K[\![x]\!]_0}$~(ii) is true for an arbitrary $M$, then the conjecture ${\bf LGF}_{K[\![x]\!]_0}$~(i) is true for an arbitrary $M$.
\end{enumerate} 
\end{thm}

\subsection{Main theorem}

Our main result of this paper is
\begin{thm}\label{thm:main}
\begin{enumerate}
\item The conjecture ${\bf LGF}_{K[\![x]\!]_0}$~(i) is true for an arbitrary $M$.
\item The conjecture ${\bf LGF}_{K[\![x]\!]_0}$~(ii) is true if the number of Frobenius slopes of $M_{\mathcal{E}}$ is less than or equal to $2$.
\end{enumerate}
\end{thm}

As mentioned in the introduction, we will study $(\sigma,\nabla)$-modules over $\mathcal{E}^{\dagger}$ rather than over $K[\![x]\!]_0$. Theorem~\ref{thm:main} will follow from Theorem~\ref{thm:main2}, which is a counterpart of Theorem~\ref{thm:main} for $(\sigma,\nabla)$-modules over $\mathcal{E}^{\dagger}$.

\section{Log-growth of analytic ring}\label{sec:lg}

In \cite{Ann} and \cite{Doc}, Kedlaya provides functorial constructions of various analytic rings associated to a certain extension of $k_K(\!(x)\!)$. We recall some of his construction. After defining a notion of log-growth on Kedlaya's analytic rings, we develop a theory of log-growth filtrations for $(\sigma,\nabla)$-modules over $\mathcal{E}^{\dagger}$.

\begin{notation}
We set $\Gamma:=\mathcal{O}_{\mathcal{E}}\subset\Gamma^{\alg}:=\mathcal{O}_{\mathcal{E}^{\sigma,\ur}}$ for compatibility with the notation in the references. We denote the norm $\nab{\cdot}{1}$ on $\Gamma[p^{-1}]$ by $\ab{\cdot}{1}$, and extend $\ab{\cdot}{1}$ to $\Gamma^{\alg}[p^{-1}]$.
\end{notation}

\begin{rem}
The ring $\Gamma^{\alg}$ in \cite{Ann}, which coincides with our $\Gamma^{\alg}$, is different from that in \cite{Doc}: the latter contains our $\Gamma^{\alg}$, but the residue field is the completion of $k_K(\!(x)\!)^{\alg}$. Fortunately, the definition of $\Gamma^{\alg}_{(\an,)\con}$ comes out the same as mentioned in \cite[2.4.13]{Doc}. By regarding our $\Gamma^{\alg}_{(\an,)\con}$ as a subring of $\Gamma^{\alg}_{(\an,)\con}$ in \cite{Doc}. we may make (careful) use of the results of \cite{Doc}.
\end{rem}

\subsection{Overconvergent rings}\label{subsec:oc}

We define subrings $\Gamma_{\con}$ and $\Gamma^{\alg}_{\con}$ of $\Gamma$ and $\Gamma^{\alg}$, respectively, as follows: For $f\in\Gamma^{\alg}[p^{-1}]$, we have a unique expression
\[
f=\sum_{i\gg -\infty}\pi_K^i[\bar{x}_i]
\]
with $\bar{x}_i\in k_K(\!(x)\!)^{\alg}$, where $[\cdot]$ denotes Teichm\"uller lift. For $n\in\N$, we define the partial valuation $v_n:\Gamma^{\alg}[p^{-1}]\to\mathbb{R}\cup\{\infty\}$ by
\[
v_n(f):=\min_{i\le n}\{v(\bar{x}_i)\},
\]
where $v$ denotes the non-archimedean valuation of $k_K(\!(x)\!)^{\alg}$ normalized by $v(x)=1$. For $r>0$, $n\in\Z$, and $f\in\Gamma^{\alg}[p^{-1}]$, we set \[
v_{n,r}(f)=r v_n(f)+n;
\]
for $r=0$, we set $v_{n,r}(f)=n$ if $v_n(f)<\infty$ and $v_{n,r}(f)=\infty$ if $v_n(f)=\infty$. For $r\in\mathbb{R}_{\ge 0}$, we denote by $\Gamma^{\alg}_r$ the subring of $f\in \Gamma^{\alg}$ such that $\lim_{n\to\infty}v_{n,r}(f)=\infty$. On $\Gamma^{\alg}_r[p^{-1}]\setminus\{0\}$, we define the non-archimedean valuation
\[
w_r(f):=\min_{n\in\mathbb{Z}}\{v_{n,r}(f)\}.
\]
Note that $\Gamma_0^{\alg}=\Gamma^{\alg}$, and $w_0$ is a $p$-adic valuation on $\Gamma^{\alg}[p^{-1}]$ normalized by $w_0(\pi_K)=1$ (\cite[2.1.11]{Doc}). We define a multiplicative norm $\ab{\cdot}{p^{-r}}:=|\pi_K|^{w_r(\cdot)}$ on $\Gamma_{r}^{\alg}[p^{-1}]$. Define $\Gamma^{\alg}_{\con}:=\cup_{r\in\mathbb{R}_{>0}}\Gamma^{\alg}_r$. Since $\Gamma_r^{\alg}\subset\Gamma_s^{\alg}$ for $0<s\le r$, we can define a value $\ab{f}{\rho}\in\mathbb{R}_{\ge 0}$ of $f\in \Gamma_{\con}^{\alg}[p^{-1}]$ for $\rho\in(0,1)$ sufficiently close to $1$ from the left. We define $\Gamma_{\con}:=\Gamma^{\alg}_{\con}\cap\Gamma$, and $\Gamma_{r}:=\Gamma^{\alg}_{r}\cap\Gamma$. Then, we have $\Gamma_{\con}=\mathcal{O}_{\mathcal{E}^{\dagger}}$ (\cite[2.3.7]{Doc}). Both $\Gamma_{\con}^{\alg}$ and $\Gamma_{\con}$ are Henselian discrete valuation rings (\cite[2.1.12, 2.2.13]{Doc}). Finally, note that $\Gamma^{\alg}_{\con}$, and hence, $\Gamma_{\con}$ is stable under $\sigma$ and $\ab{\sigma(\cdot)}{\rho}=\ab{\cdot}{\rho^q}$ for $\rho\in (0,1)$.

\begin{dfn}[{\cite[3.5]{Ann}}]\label{dfn:NP}
Let $f\in\Gamma^{\alg}_{r}[p^{-1}]$ be a non-zero element. We define the Newton polygon $\NP(f)$ of $f$ as the boundary of the lower convex hull of the set of points $(v_n(f),n)$, minus any segments of slopes less than $-r$ from the left end and/or any segments of non-negative slope on the right end of the polygon. We define the slopes of $f$ as the negatives of the slopes of $\NP(f)$. We also define the multiplicity of a slope $s\in (0,r]$ of $f$ as the positive difference in $y$-coordinate between the endpoints of the segment of $\NP(f)$ of slope $-s$.
\end{dfn}

The following simple fact is one of the key points in this paper.

\begin{lem}[{cf. \cite[Lemma~2.6]{Nak}}]\label{lem:lin}
Let $f\in\Gamma^{\alg}_{\con}[p^{-1}]$. Then, there exists $\rho_0\in\mathbb{R}_{>0}$ and $a\in\mathbb{Q}$ such that
\[
\ab{f}{\rho}={\rho}^a\ab{f}{1}\text{ for all }\rho\in (\rho_0,1].
\]
\end{lem}
\begin{proof}
We may assume $f\neq 0$ and $f\in\Gamma_{r}^{\alg}[p^{-1}]$ for some $r>0$. Since the number of the slopes of $\NP(f)$ with non-zero multiplicities is finite by \cite[2.4.6]{Doc}, we may assume that $f$ has no slopes after choosing $r$ sufficiently small. By \cite[2.4.6]{Doc} again, there exists a unique integer $n$ such that $w_s(f)=v_{n,s}(f)$ for all $s\in [0,r]$. Then, we have $\ab{f}{p^{-s}}=(p^{-s})^{v_n(f)/e_K}|\pi_K|^n$ for any $s\in [0,r]$, where $e_K$ is the absolute ramification index of $K$. By evaluating $s=0$, $\ab{f}{1}=|\pi_K|^n$. Hence, we obtain the assertion for $\rho_0=p^{-r}$ and $a=v_n(f)/e_K$.
\end{proof}

\subsection{Log-growth filtration over $\Gamma_{\con}[p^{-1}]$}\label{subsec:fil}

Throughout this section, let $\bullet$ denote either \ (blank) or alg. Let $\Gamma^{\bullet}_{\an,r}$ be the Fr\'echet completion of the ring $\Gamma^{\bullet}_r[p^{-1}]$ with respect to the family of valuations $\{w_s\}_{s\in (0,r]}$ (\cite[3.3]{Ann}). We define $\Gamma^{\bullet}_{\an,\con}:=\cup_{r\in\mathbb{R}_{>0}}\Gamma^{\bullet}_{\an,r}$. Then, we have $\Gamma_{\an,\con}=\mathcal{R}$, in particular, $\Gamma_{\an,\con}$ contains $K\{x\}$. By continuity, $\Gamma^{\bullet}_{\an,r}$ is endowed with a family of non-archimedean valuations induced by $\{v_n\}_{n\in\mathbb{Z}}$ and $\{w_s\}_{s\in (0,r]}$. In addition, the norm $\ab{\cdot}{p^{-r}}$ extends to $\Gamma_{\an,r}^{\bullet}$. As before, we can define a value $\ab{f}{\rho}\in\mathbb{R}_{\ge 0}$ of $f\in\Gamma_{\an,\con}^{\bullet}$ for $\rho\in(0,1)$ sufficiently close to $1$ from the left.

\begin{rem}
As mentioned above, we have $\Gamma_{\con}[p^{-1}]=\mathcal{E}^{\dagger}$ and $\Gamma_{\an,\con}=\mathcal{R}$ as rings. However, the partial norms $\ab{\cdot}{\rho}$ on $\Gamma_{\an,\con}$ and $|\cdot|^{\naive}_0(\rho)$ on $\mathcal{R}$ coincide with each other only when $\rho$ is sufficiently close to $1$ (\cite[2.3.5]{Doc}). For this reason, we will distinguish $\Gamma_{\con}[p^{-1}]$ and $\Gamma_{\an,\con}$ from $\mathcal{E}^{\dagger}$ and $\mathcal{R}$, respectively, as normed rings.
\end{rem}

\begin{dfn}[{Log-growth of analytic ring (cf. \cite[2.8]{Nak})}]
For $\lambda\in\mathbb{R}$, we denote by $\Fil_{\lambda}\Gamma^{\bullet}_{\an,\con}$ the subspace of $f\in\Gamma^{\bullet}_{\an,\con}$ such that
\[
\ab{f}{\rho}=\RO{-\lambda}\text{ as }\rho\uparrow 1.
\]
\end{dfn}

\begin{lem}\label{lem:an}
We have the following:
\begin{enumerate}
\item For a non-zero $f\in\Gamma^{\bullet}_{\an,\con}$,
\[
\liminf_{\rho\uparrow 1}\ab{f}{\rho}>0.
\]
\item
\[
\Fil_0\Gamma_{\an,\con}=\Gamma_{\con}[p^{-1}],\ \Fil_0\Gamma^{\alg}_{\an,\con}\supset\Gamma^{\alg}_{\con}[p^{-1}].
\]
\item
\[
K\{x\}\cap \Fil_{\lambda}\Gamma_{\an,\con}=K[\![x]\!]_{\lambda}\text{ for }\lambda\in\mathbb{R}.
\]
\item
\[
\sigma(\Fil_{\lambda}\Gamma^{\bullet}_{\an,\con})\subset \Fil_{\lambda}\Gamma^{\bullet}_{\an,\con}\text{ for }\lambda\in\mathbb{R}.
\]
\item
\[
\Fil_{\lambda_1}\Gamma^{\bullet}_{\an,\con}\cdot \Fil_{\lambda_2}\Gamma^{\bullet}_{\an,\con}\subset \Fil_{\lambda_1+\lambda_2}\Gamma^{\bullet}_{\an,\con}\text{ for }\lambda_1,\lambda_2\in\mathbb{R}.
\]
\end{enumerate}
\end{lem}
\begin{proof}
\begin{enumerate}
\item We choose $r>0$ sufficiently small such that $f\in\Gamma_{\an,r}$. Then, we have $w_r(f)\neq \infty$ because $f\neq 0$. In particular, there exists $n\in\mathbb{Z}$ such that $v_n(f)\neq\infty$. By definition, $w_s(f)\le sv_n(f)+n$ for all $s\in (0,r]$. Therefore, $\limsup_{s\downarrow 0}w_s(f)\le n<\infty$, which implies the assertion.
\item By Lemma~\ref{lem:lin}, we have only to prove that $f\in\Fil_0\Gamma_{\an,\con}$ belongs to $\Gamma_{\con}[p^{-1}]$. Since $\ab{f}{\rho}=O(1)$ as $\rho\uparrow 1$, there exist a constant $C$ and $r>0$ such that $C<w_s(f)$ for all $s\in (0,r]$. If $v_n(f)<\infty$, then we have $C<n$ by taking the limit $s\downarrow 0$ in the inequality $w_s(f)\le s v_n(f)+n$. Hence, we have $v_n(f)=\infty$ for all sufficiently small $n\in\mathbb{Z}$. If we take $l\in\mathbb{Z}$ such that $v_n(\pi_K^lf)=\infty$ for all $n<0$, then we have $\pi_K^lf\in\Gamma_r$ by \cite[2.5.6]{Doc}, i.e., $f\in\Gamma_{\con}[\pi_K^{-1}]=\Gamma_{\con}[p^{-1}]$.
\item It follows from the fact that for $f\in K\{x\}$, we have $\ab{f}{\rho}=\nab{f}{\rho}$ for $\rho$ sufficiently close to $1$ from the left (\cite[2.3.5]{Doc}).
\item It follows from $\ab{\sigma(\cdot)}{\rho}=\ab{\cdot}{\rho^q}$.
\item The assertion follows from the multiplicativity of the norm $\ab{\cdot}{\rho}$.
\end{enumerate}
\end{proof}

Note that $\Fil_{\lambda}\Gamma^{\bullet}_{\an,\con}=0$ for $\lambda\in\mathbb{R}_{<0}$ by Lemma~\ref{lem:an}~(i). In addition, Lemma~\ref{lem:an} implies that $\Fil_{\lambda}\Gamma_{\an,\con}^{\bullet}$ forms an increasing filtration of $\sigma$-stable $\Gamma_{\con}^{\bullet}[p^{-1}]$-subspaces of $\Gamma_{\an,\con}^{\bullet}$.

\begin{rem}
In (i), the equality in the latter case does not hold. Indeed, there exists $f\in \Gamma^{\alg}_{\an,\con}$ such that $v_n(f)=\infty$ for $n\in\mathbb{Z}_{<0}$, but $f\notin  \Gamma^{\alg}_{\con}$ (\cite[2.4.13]{Doc}).
\end{rem}

\begin{dfn}[{A log extension of $\Gamma_{\an,\con}^{\bullet}$ (\cite[6.5]{Ann})}]
We set $\Gamma^{\bullet}_{\log,\an,\con}:=\Gamma^{\bullet}_{\an,\con}[\log{x}]$, where $\log{x}$ is an indeterminate. We can extend $\sigma$ to $\Gamma^{\bullet}_{\log,\an,\con}$ as follows:
\[
\sigma(\log{x}):=q\log{x}+\sum_{i=1}^{\infty}\frac{(-1)^{i-1}}{i}\left(\frac{\sigma(x)}{x^q}-1\right)^i.
\]
Moreover, we extend $d/dx$ to $\Gamma_{\log,\an,\con}=\mathcal{R}[\log{x}]$ by
\[
\frac{d}{dx}(\log{x})=\frac{1}{x}.
\]
We also define the notion of $(\sigma,\nabla)$-modules over $\Gamma_{\log,\an,\con}$ by setting $R=\Gamma_{\log,\an,\con}$ and $\Omega^1_R=\Gamma_{\log,\an,\con}dx$ in \S~\ref{subsec:lgfil}.

For $\rho\in (0,1)$, we put $r:=-\log_p{\rho}$ and extend $\ab{\cdot}{\rho}$ to $\Gamma^{\bullet}_{\an,r}[\log{x}]$ by
\[
\ab{\sum_{i\in\mathbb{N}}a_i(\log{x})^i}{\rho}:=\sup_{i\in\mathbb{N}}\ab{a_i}{\rho}\cdot\Ro{-i}.
\]
\end{dfn}

\begin{lem}
The function $\ab{\cdot}{\rho}$ is a multiplicative non-archimedean norm on the ring $\Gamma^{\bullet}_{\an,r}[\log{x}]$.
\end{lem}
\begin{proof}
We have only to check the multiplicativity of $\ab{\cdot}{\rho}$. Let $f=\sum_ia_i(\log{x})^i,g=\sum_jb_j(\log{x})^j\in\Gamma^{\bullet}_{\an,r}[\log{x}]$. We have $\ab{fg}{\rho}\le \ab{f}{\rho}\cdot \ab{g}{\rho}$ by definition. We prove the converse. We may assume $f\neq 0$ and $g\neq 0$. Let $i_0$ (resp. $j_0$) be the minimum $i$ (resp. $j$) such that $\ab{f}{\rho}=\ab{a_i}{\rho}\cdot (\log{1/\rho})^{-i}$ (resp. $\ab{g}{\rho}=\ab{b_j}{\rho}\cdot (\log{1/\rho})^{-j}$). For $i_1<i_0$ and $j_0\le j_1$, we have
\[
\ab{a_{i_0}}{\rho}\cdot\Ro{-i_0}>\ab{a_{i_1}}{\rho}\cdot\Ro{-i_1},\ \ab{b_{j_0}}{\rho}\cdot\Ro{-j_0}\ge\ab{b_{j_1}}{\rho}\cdot\Ro{-j_1},
\]
and hence, $\ab{a_{i_0}b_{j_0}(\log{x})^{i_0+j_0}}{\rho}>\ab{a_{i_1}b_{j_1}(\log{x})^{i_1+j_1}}{\rho}$. Similarly, we have
\[
\ab{a_{i_0}b_{j_0}(\log{x})^{i_0+j_0}}{\rho}>\ab{a_{i_1}b_{j_1}(\log{x})^{i_1+j_1}}{\rho}
\]
for $i_1\ge i_0$ and $j_0>j_1$. Therefore, we have
\[
\ab{fg}{\rho}\ge \ab{\sum_{i+j=i_0+j_0}{a_ib_j}(\log{x})^{i+j}}{\rho}=\ab{f}{\rho}\cdot\ab{g}{\rho}.
\]
\end{proof}

\begin{dfn}\label{dfn:fil2} 
We define a log-growth filtration of $\Gamma^{\bullet}_{\log,\an,\con}$ by
\[
\Fil_{\lambda}\Gamma^{\bullet}_{\log,\an,\con}:=\bigoplus_{i=0}^{\lfloor \lambda\rfloor}\Fil_{\lambda-i}\Gamma^{\bullet}_{\an,\con}\cdot(\log{x})^i
\]
for $\lambda\in\mathbb{R}_{\ge 0}$ and $\Fil_{\lambda}\Gamma^{\bullet}_{\log,\an,\con}:=0$ for $\lambda\in\mathbb{R}_{<0}$. Here, $\lfloor \lambda\rfloor$ denotes the greatest integer less than or equal to $\lambda$. For $\lambda\in\mathbb{R}$, we say that $y\in\Gamma^{\bullet}_{\log,\an,\con}$ has a log-growth $\lambda$ if $y\in \Fil_{\lambda}\Gamma^{\bullet}_{\log,\an,\con}$. Moreover, we say that $f$ is bounded if $f$ has a log-growth $0$. For $\lambda\in\mathbb{R}_{>0}$, we also say that $y$ is exactly of log-growth $\lambda$ if $y\in \Fil_{\lambda}\Gamma^{\bullet}_{\log,\an,\con}$ and $y\notin \Fil_{\delta}\Gamma^{\bullet}_{\log,\an,\con}$ for any $0\le \delta<\lambda$ (\cite[1.1]{CT}).
\end{dfn}

\begin{lem}\label{lem:logan}
\begin{enumerate}
\item For $f\in\Gamma^{\bullet}_{\log,\an,\con}$ and $\lambda\in\mathbb{R}$, we have $f\in \Fil_{\lambda}\Gamma^{\bullet}_{\log,\an,\con}$ if and only if
\[
\ab{f}{\rho}=\RO{-\lambda}\text{ as }\rho\uparrow 1.
\]
\item
\[
\sigma(\Fil_{\lambda}\Gamma^{\bullet}_{\log,\an,\con})\subset \Fil_{\lambda}\Gamma^{\bullet}_{\log,\an,\con}\text{ for }\lambda\in\mathbb{R}.
\]
\item
\[
\Fil_{\lambda_1}\Gamma^{\bullet}_{\log,\an,\con}\cdot \Fil_{\lambda_2}\Gamma^{\bullet}_{\log,\an,\con}\subset \Fil_{\lambda_1+\lambda_2}\Gamma^{\bullet}_{\log,\an,\con}\text{ for }\lambda_1,\lambda_2\in\mathbb{R}.
\]
\end{enumerate}
\end{lem}
\begin{proof}
\begin{enumerate}
\item The assertion  follows from the definition.
\item The assertion follows from $\sum_{i=1}^{\infty}\frac{(-1)^{i-1}}{i}\left(\frac{\sigma(x)}{x^q}-1\right)^i\in\Gamma_{\con}[p^{-1}]=\Fil_0\Gamma_{\an,\con}$ (\cite[6.5]{Ann}).
\item The assertion follows from Lemma~\ref{lem:an}~(v).
\end{enumerate}
\end{proof}

\begin{dfn}[Log-growth filtration]\label{dfn:genfil}
Let $M$ be a $(\sigma,\nabla)$-module of rank~$n$ over $\Gamma_{\con}[p^{-1}]$. We set
\[
\mathfrak{V}(M):=(\Gamma_{\log,\an,\con}\otimes_{\Gamma_{\con}[p^{-1}]}M)^{\nabla=0},
\]
\[
\mathfrak{Sol}(M):=\Hom_{\Gamma_{\con}[p^{-1}]}(M,\Gamma_{\log,\an,\con})^{\nabla=0}\cong \mathfrak{V}(M\spcheck).
\]
We say that $M$ is solvable in $\Gamma_{\log,\an,\con}$ if $\dim_K\mathfrak{V}(M)=n$. In this case, we define
\[
\mathfrak{Sol}_{\lambda}(M):=\Hom_{\Gamma_{\con}[p^{-1}]}(M,\Fil_{\lambda}\Gamma_{\log,\an,\con})\cap\mathfrak{Sol}(M)
\]
and
\[
\mathfrak{V}(M)^{\lambda}:=\mathfrak{Sol}_{\lambda}(M)^{\perp}
\]
where $(\cdot)^{\perp}$ denotes the orthogonal space with respect to the canonical pairing $\mathfrak{V}(M)\otimes_K\mathfrak{Sol}(M)\to K$. We call the decreasing filtration $\{\mathfrak{V}(M)^{\lambda}\}_{\lambda}$ the special log-growth filtration of $M$.

\end{dfn}

Note that if $M$ is a $(\sigma,\nabla)$-module over $\Gamma_{\con}[p^{-1}]$ solvable in $\Gamma_{\log,\an,\con}$, then $\mathfrak{V}(M)$ is a $\sigma$-module over $K$ by the injectivity of $\varphi:\mathfrak{V}(M)\to\mathfrak{V}(M)$. By Lemma~\ref{lem:logan}~(ii), $\mathfrak{V}(M)^{\lambda}$ (resp. $\mathfrak{Sol}_{\lambda}(M)$) is a $\sigma$-submodule of $\mathfrak{V}(M)$ (resp. $\mathfrak{Sol}(M)$).

\begin{example}
We give an example of a $\nabla$-module defined over $\Gamma_{\con}[p^{-1}]$, which is not necessary defined over $K[\![x]\!]_0$. Let $a=\sum_{n\in\mathbb{Z}}a_nx^n\in\Gamma_{\con}[p^{-1}]$ with $a_n\in K$ and we assume that there exists $\delta>0$ such that
\[
O(|a_{-n}|)=O(p^{-n\delta})\text{ as }n\to\infty.
\]
For example, assume $a_{-n}=0$ for all $n\gg 0$. Let $M=\Gamma_{\con}[p^{-1}]e_1\oplus\Gamma_{\con}[p^{-1}]e_2$ be the $\nabla$-module of rank $2$ over $\Gamma_{\con}[p^{-1}]$ defined by
\[
\nabla(e_1,e_2)=(0,ae_1dx).
\]
We set
\[
f=\sum_{n\neq -1}\frac{1}{n+1}a_nx^{n+1}+a_{-1}\log{x},
\]
which belongs to $\Gamma_{\log,\an,\con}$ as follows. For $n\ge 0$ and $\rho\in [0,1)$,
\[
\left|\frac{1}{n+1}a_nx^{n+1}\right|_{\rho}\le\frac{1}{n+1}\cdot\sup_{n\ge 0}|a_n|\cdot\rho^{n+1}\to 0\text{ as }n\to\infty.
\]
We fix $\rho_0\in (p^{-\delta},1)$. For $\rho\in [\rho_0,1)$ and $n\ge 2$,
\[
\left|\frac{1}{-n+1}a_{-n}x^{-n+1}\right|_{\rho}\le\frac{1}{n-1}\cdot\frac{|a_{-n}|}{p^{-n\delta}}\cdot\left(\frac{1}{p^{\delta}\rho}\right)^n\cdot\rho\le\frac{1}{n-1}\cdot\frac{|a_{-n}|}{p^{-n\delta}}\cdot\left(\frac{1}{p^{\delta}\rho_0}\right)^n\to 0\text{ as }n\to\infty.
\]
Since the convergence is uniform with respect to $\rho$ in the latter inequality, we have $f\in\Fil_{\xi}\Gamma_{\log \an,\con}$, where we define $\xi$ as the log-growth of the power series $\sum_{n\ge 0}\frac{1}{n+1}a_nx^n$ if $a_{-1}=0$, and $\xi=1$ if $a_{-1}\neq 0$. Then, $\mathfrak{V}(M):=(\Gamma_{\log,\an,\con}\otimes_{\Gamma_{\con}[p^{-1}]}M)^{\nabla=0}$ has a basis $\{e_1,fe_1-e_2\}$ and we have
\[
\mathfrak{V}(M)^{\lambda}=
\begin{cases}
\mathfrak{V}(M)&\text{if }\lambda<0\\
Ke_1&\text{if }0\le\lambda<\xi\\
0&\text{if }\xi\le\lambda.
\end{cases}
\]
\end{example}

\begin{rem}
We can define a special log-growth filtration for an arbitrary $(\sigma,\nabla)$-module over $\Gamma_{\con}[p^{-1}]$ possibly after tensoring a suitable extension of $\Gamma_{\con}[p^{-1}]$ as follows. Let $M$ be a $(\sigma,\nabla)$-module over $\Gamma_{\con}[p^{-1}]$. Then, there exists a finite \'etale extension $\Gamma^l/\Gamma$, corresponding to a certain finite separable extension $l/k_K(\!(x)\!)$, such that $M':=\Gamma_{\con}^l[p^{-1}]\otimes_{\Gamma_{\con}[p^{-1}]}M$ is solvable in $\Gamma^l_{\log,\an,\con}$ by the log version of the $p$-adic local monodromy theorem (\cite[6.13]{Ann}). Similarly as above, we may define a special log-growth filtration of $M'$.
\end{rem}

The log-growth filtrations are compatible with the base change $K[\![x]\!]_0\to\mathcal{E}^{\dagger}=\Gamma_{\con}[p^{-1}]$:

\begin{lem}\label{lem:bc}
Let $M$ be a $(\sigma,\nabla)$-module over $K[\![x]\!]_0$. Then, the $(\sigma,\nabla)$-module $\Gamma_{\con}[p^{-1}]\otimes_{K[\![x]\!]_0}M$ over $\Gamma_{\con}[p^{-1}]$ is solvable in $\Gamma_{\log,\an,\con}$. Moreover, the canonical map
\[
\iota:V(M)\to\mathfrak{V}(\Gamma_{\con}[p^{-1}]\otimes_{K[\![x]\!]_0}M)
\]
is an isomorphism, and preserves the Frobenius filtrations and the log-growth filtrations.
\end{lem}
\begin{proof}
Since the natural inclusion $K\{x\}\to\Gamma_{\log,\an,\con}$ is compatible with Frobenius and differentials, $\Gamma_{\con}[p^{-1}]\otimes_{K[\![x]\!]_0}M$ is solvable in $\Gamma_{\log,\an,\con}$, and $\iota$ is an isomorphism of $\sigma$-modules over $K$. The rest of the assertion follows from $\Fil_{\lambda}\Gamma_{\log,\an,\con}\cap K\{x\}=\Fil_{\lambda}\Gamma_{\an,\con}\cap K\{x\}=K[\![x]\!]_{\lambda}$ (Lemma~\ref{lem:an}~(iii)).
\end{proof}

\subsection{Chiarellotto-Tsuzuki's conjecture over $\Gamma_{\con}[p^{-1}]$}

We formulate an analogue of Theorem~\ref{thm:main} for $(\sigma,\nabla)$-modules over $\Gamma_{\con}[p^{-1}]$.

\begin{assumption}\label{ass:alg}
In this section, we assume that $k_K$ is algebraically closed for simplicity.
\end{assumption}

\begin{dfn}
Let $M$ be a $(\sigma,\nabla)$-module over $\Gamma_{\con}[p^{-1}]$ solvable in $\Gamma_{\log,\an,\con}$. 
\begin{enumerate}
\item We call a Frobenius slope of $\mathfrak{V}(M)$ a special Frobenius slope of $M$.
\item We set $M_{\mathcal{E}}:=\mathcal{E}\otimes_{\Gamma_{\con}[p^{-1}]}M$, which is a $(\sigma,\nabla)$-module over $\mathcal{E}$. We call a Frobenius slope of $M_{\mathcal{E}}$ a generic Frobenius slope of $M$.
\end{enumerate}
\end{dfn}

\begin{prop}\label{prop:basic}
Let $M$ be a $(\sigma,\nabla)$-module over $\Gamma_{\con}[p^{-1}]$ solvable in $\Gamma_{\log,\an,\con}$. Let $\lambda_{\max}$ be the highest Frobenius slope of $M_{\mathcal{E}}$.
\begin{enumerate}
\item (Analogue of \cite[Theorem 6.17]{CT}) We have
\[
\mathfrak{V}(M)^{\lambda}\subset (S_{\lambda-\lambda_{\max}}(\mathfrak{V}(M\spcheck)))^{\perp}
\]
for all $\lambda\in\mathbb{R}$. Here, $(\cdot)^{\perp}$ denotes the orthogonal space with respect to the canonical pairing $\mathfrak{V}(M)\otimes_K\mathfrak{V}(M\spcheck)\to K$.
\item If $M_{\mathcal{E}}$ is PBQ, then
\[
\mathfrak{V}(M)^0=(S_{-\lambda_{\max}}(\mathfrak{V}(M\spcheck)))^{\perp}.
\]
\end{enumerate}
\end{prop}

\begin{thm}[{Generalization of Theorem~\ref{thm:main}}]\label{thm:main2}
Let $M$ be a $(\sigma,\nabla)$-module over $\Gamma_{\con}[p^{-1}]$ solvable in $\Gamma_{\log,\an,\con}$.
\begin{enumerate}
\item The special log-growth filtration of $M$ is rational and right continuous .
\item Let $\lambda_{\max}$ be the highest Frobenius slope of $M_{\mathcal{E}}$. Assume that $M_{\mathcal{E}}$ is PBQ and the number of the Frobenius slopes of $M_{\mathcal{E}}$ is less than or equal to $2$. Then,
\[
\mathfrak{V}(M)^{\lambda}=(S_{\lambda-\lambda_{\max}}(\mathfrak{V}(M\spcheck)))^{\perp}
\]
for all $\lambda\in\mathbb{R}$.
\end{enumerate}
\end{thm}

Recall that we may assume Assumption~\ref{ass:alg} to prove Conjecture~\ref{conj:sp} (\S~\ref{subsec:CT}). Hence, Theorem~\ref{thm:main} follows from Theorem~\ref{thm:main2} by Lemma~\ref{lem:bc}. The proofs of Proposition~\ref{prop:basic} and Theorem~\ref{thm:main2} will be given in \S~\ref{sec:pf}.

\begin{rem}
Obviously, one can formulate an analogue of Conjecture~\ref{conj:sp}~(ii) for a $(\sigma,\nabla)$-module over $\Gamma_{\con}[p^{-1}]$ such that $M_{\mathcal{E}}$ is PBQ.
\end{rem}

\subsection{Example: $p$-adic differential equations with nilpotent singularities}\label{subsec:nil}

In \cite[\S~18]{pde}, Kedlaya studies effective bounds on the solutions of $p$-adic differential equations with nilpotent singularities. As an application, he proves a nilpotent singular analogue of Theorem~\ref{thm:CT}~(iii) (\cite[Remark~18.4.4, Theorem~18.4.5]{pde}). In this subsection, we explain that a nilpotent singular analogue of Theorem~\ref{thm:main} follows from Theorem~\ref{thm:main2}.

In the following, we assume $\sigma(x)=x^q$. We define $\Omega^1_{K[\![x]\!]_0}(\log)$ as a $\sigma$-module of rank~$1$ over $K[\![x]\!]_0$ with basis $dx/x$ such that $\sigma^*(1\otimes dx/x):=qdx/x$. Let
\[
d:K[\![x]\!]_0\to \Omega^1_{K[\![x]\!]_0}(\log)=K[\![x]\!]_0dx/x;f\mapsto xdf/dx\cdot dx/x
\]
be the canonical derivation on $K[\![x]\!]_0$. We define a log $(\sigma,\nabla)$-module over $K[\![x]\!]_0$ similarly to \S~\ref{subsec:lgfil} by setting $R=K[\![x]\!]_0$ and $\Omega^1_R=\Omega^1_{K[\![x]\!]_0}(\log)$. 

As in Definition~\ref{dfn:fil2}, we define a log-growth filtration of $K[\![x]\!][\log{x}]$ as
\[
K[\![x]\!][\log{x}]_{\lambda}:=\bigoplus_{i=0}^{\lfloor\lambda\rfloor}K[\![x]\!]_{\lambda-i}(\log{x})^i
\]
for $\lambda\in\mathbb{R}_{\ge 0}$ and $K[\![x]\!][\log{x}]_{\lambda}:=0$ for $\lambda\in\mathbb{R}_{<0}$. For a log $(\sigma,\nabla)$-module $M$ over $K[\![x]\!]_0$, we define
\[
V(M):=(K\{x\}[\log{x}]\otimes_{K[\![x]\!]_0}M)^{\nabla=0}.
\]
By Dwork's trick, $V(M)$ is of dimension $n$ (\cite[Corollary~17.2.4]{pde}). We define a special log-growth filtration $V(M)^{\bullet}$ of $M$ as in \S~\ref{subsec:fil} by replacing $K[\![x]\!]_{\lambda}$ by $K[\![x]\!][\log{x}]_{\lambda}$.

\begin{example}
\begin{enumerate}
\item A $(\sigma,\nabla)$-module over $K[\![x]\!]_0$ can be regarded as a log $(\sigma,\nabla)$-module over $K[\![x]\!]_0$ by identifying $dx$ as $x\cdot dx/x$. The special log-growth filtration of $M$ as a non-log or log $(\sigma,\nabla)$-module coincides with each other.
\item Let $M:=K[\![x]\!]_0e_1\oplus K[\![x]\!]_0e_2$ be the log $(\sigma,\nabla)$-module of rank~$2$ over $K[\![x]\!]_0$ defined by
\[
\nabla(e_1,e_2)=(0,e_1dx/x),\ \varphi(e_1,e_2)=(e_1,qe_2).
\]
Then, $V(M)$ has a basis $\{e_1,-\log{x}\cdot e_1+e_2\}$. Moreover, the Frobenius slopes of $V(M)$ are $0,1$, and we have
\[
V(M)^{\lambda}=
\begin{cases}
V(M)&\text{if }\lambda<0\\
Ke_1&\text{if }0\le\lambda<1\\
0&\text{otherwise.}
\end{cases}
\]
\end{enumerate}
\end{example}

Our main result in this subsection is
\begin{thm}
An analogue of Theorem~\ref{thm:main} for log $(\sigma,\nabla)$-modules over $K[\![x]\!]_0$ holds.
\end{thm}
\begin{proof}
It follows from Theorem~\ref{thm:main2} thanks to Lemma~\ref{lem:bc2} below.
\end{proof}

\begin{lem}\label{lem:bc2}
Let $M$ be a log $(\sigma,\nabla)$-module over $K[\![x]\!]_0$. Then, the $(\sigma,\nabla)$-module $\Gamma_{\con}[p^{-1}]\otimes_{K[\![x]\!]_0}M$ over $\Gamma_{\con}[p^{-1}]$ is solvable in $\Gamma_{\log,\an,\con}$. Moreover, the canonical map
\[
\iota:V(M)\to\mathfrak{V}(\Gamma_{\con}[p^{-1}]\otimes_{K[\![x]\!]_0}M)
\]
is an isomorphism, and preserves the Frobenius filtrations and log-growth filtrations.
\end{lem}
\begin{proof}
Similar to the proof of Lemma~\ref{lem:bc}.
\end{proof}

\section{Generic cyclic vector}\label{sec:gen}

In this section, we prove a key technical result in this paper concerned with a $\sigma$-module over $\Gamma_{\con}^{\alg}[p^{-1}]$.

\begin{dfn}\label{dfn:cyc}
Let $R$ be either $\Gamma_{\con}^{\alg}[p^{-1}]$ or $\Gamma^{\alg}[p^{-1}]$.
\begin{enumerate}
\item For $f(\sigma)=a_0+a_1\sigma+\dots+a_n\sigma^n\in R\{\sigma\}$ a twisted polynomial, we define the Newton polygon $\NP(f(\sigma))$ of $f(\sigma)$ as the boundary of the lower convex hull of the set of points
\[
\{(i,-\log_q{\ab{a_i}{1}});0\le i\le n\}.
\]
A slope of $\NP(f(\sigma))$ is called a slope of $f(\sigma)$ (cf. \cite[2.1.3]{pde}; note that the Newton polygon in \cite{pde} is defined as the boundary of the lower convex hull of $\{(-i,-\log_q{\ab{a_i}{1}});0\le i\le n\}$. Consequently, our slopes coincide with $-1$ times the slopes in \cite{pde}). We consider the following condition $(*)$ on $f(\sigma)$:
\[
(*):\text{ each point }(i,-\log_q{\ab{a_i}{1}})\text{ belongs to }\NP(f(\sigma)).
\]
\item Let $M$ be a $\sigma$-module of rank $n$ over $R$. When $R=\Gamma_{\con}^{\alg}[p^{-1}]$, we call a Frobenius slope (resp. the Newton polygon) of $\Gamma_{\con}^{\alg}[p^{-1}]\otimes_RM$ a generic Frobenius slope (resp. the generic Newton polygon) of $M$. We say that an element $e\in M$ is a cyclic vector if $e,\varphi(e),\dots,\varphi^{n-1}(e)$ is a basis of $M$ over $R$. For a cyclic vector $e$, we have a unique relation
\[
\varphi^n(e)=-(a_{n-1}\varphi^{n-1}(e)+\dots+a_0 e)
\]
with $a_i\in R$. We set $f_e(\sigma):=a_0+a_{1}\sigma+\dots+\sigma^n\in R\{\sigma\}$. Note that $\NP(f_e(\sigma))$ coincides with the (generic) Frobenius Newton polygon of $M\spcheck$ (\cite[14.5.7]{pde}).

We say that a cyclic vector $e\in M$ is generic if $f_e(\sigma)$ satisfies the condition (*).
\end{enumerate}
\end{dfn}

\begin{thm}\label{thm:cyc}
Let $M$ be a $\sigma$-module over $\Gamma^{\alg}[p^{-1}]$ (resp. $\Gamma_{\con}^{\alg}[p^{-1}]$). Assume $q^{s}\in \mathbb{Q}$ for any (resp. generic) Frobenius slope $s$ of $M$. Then, there exists a generic cyclic vector of $M$.
\end{thm}

In the next subsection, we see that there exists a non-empty open subset $U$ of $M$ such that $v\in U$ is a generic cyclic vector. In this sense, there exist a number of cyclic vectors satisfying the condition (*). Therefore, the condition (*) is referred to as being generic.

\subsection{Proof of Theorem~\ref{thm:cyc}}

To prove Theorem~\ref{thm:cyc}, we first construct a generic cyclic vector over $\Gamma^{\alg}[p^{-1}]$. Then, we deform it to obtain a generic cyclic vector over $\Gamma^{\alg}_{\con}[p^{-1}]$. We first recall Kedlaya's algorithm to compute an annihilator of an element of a $\sigma$-module over $\Gamma^{\alg}[p^{-1}]$ (\cite[5.2.4]{Doc}).

\begin{construction}\label{const:Ked}
Let $R:=\Gamma^{\alg}[p^{-1}]$. Let $M$ be a $\sigma$-module of rank $n$ over $R$ with Frobenius slopes $\s{1}\le \dots\le \s{n}$ with multiplicities. Assume $q^{\s{i}}\in\mathbb{Q}$ for all $i$ (\S~\ref{subsec:phi}). Then, we can choose an $R$-basis $e_1,\dots,e_n$ of $M$ such that $\varphi(e_i)=q^{\s{i}} e_i$ for all $i$. Fix $x_1\dots,x_n\in R$ and set $v:=x_1e_1+\dots+x_ne_n$. We define $v_l\in M$ for $1\le l\le n$ by induction on $l$. Set $v_1:=v$. Given $v_l$, write $v_l=x_{l,1}e_1+\dots+x_{l,n}e_n$ with $x_{l,i}\in R$ and define
\[
b_l:=
\begin{cases}
q^{\s{l}}\cdot \sigma(x_{l,l})/x_{l,l}&\text{if }x_{l,l}\neq 0\\
0&\text{otherwise}
\end{cases}
\]
and $v_{l+1}:=(\varphi-b_l)v_l$. Then, we have $v_l\in Re_{l}+\dots+Re_n$ and $(\varphi-b_{n})\dots (\varphi-b_1)v=0$. We write
\[
(\sigma-b_{n})\dots (\sigma-b_1)=\sigma^n+c_{n-1}\sigma^{n-1}+\dots+c_0,\ c_i\in R
\]
in $R\{\sigma\}$. By construction, we may regard $x_{l,i}=x_{l,i}(\bvec{x})$, $b_l=b_l(\bvec{x})$, and $c_i=c_i(\bvec{x})$ as functions of $\bvec{x}=(x_1,\dots,x_n)\in R^n$ with values in $R$. We also regard $v=v(\bvec{x})$ as a function of $\bvec{x}$ with values in $M$.
\end{construction}

\begin{lem}\label{lem:conti}
We retain the notation in Construction~\ref{const:Ked}.
\begin{enumerate}
\item For $\bvec{x}\in R^n$, $v(\bvec{x})$ is a cyclic vector of $M$ if and only if $x_{1,1}(\bvec{x})x_{2,2}(\bvec{x})\dots x_{n,n}(\bvec{x})\neq 0$.
\item For $\bvec{x}\in R^n$, $v(\bvec{x})$ is a generic cyclic vector of $M$ if and only if $x_{1,1}(\bvec{x})x_{2,2}(\bvec{x})\dots x_{n,n}(\bvec{x})\neq 0$ and $-\log_q{\ab{c_i(\bvec{x})}{1}}=\s{1}+\dots+\s{n-i}$ for all $i$.
\item Let $\bvec{x}^{(0)}\in R^n$. Assume that $b_1(\bvec{x}^{(0)}),\dots,b_{n}(\bvec{x}^{(0)})$ are all non-zero. Then, there exists an open neighborhood $U\subset R^n$ of $\bvec{x}^{(0)}$ (with respect to the topology induced by $\ab{\cdot}{1}$) such that all $b_l$ and $x_{l,i}$ are continuous on $U$. In particular, all $c_i$ are also continuous on $U$.
\end{enumerate}
\end{lem}
\begin{proof}
\begin{enumerate}
\item By construction, there exists an upper triangular matrix $T$ whose diagonals are $(1,\dots,1)$ such that $(v_1,v_2,\dots,v_{n})=(v,\varphi(v),\dots,\varphi^{n-1}(v))T$. Since $\{x_{l,i}\}_{l,i}$ is an upper triangular matrix, we obtain the assertion.
\item It follows from (i) and the fact that the slopes of $\sigma^n+c_{n-1}\sigma^{n-1}+\dots+c_0$ are $-\s{n}\le\dots\le -\s{1}$ with multiplicities.
\item By induction on $l\in\{1,\dots,n\}$, we construct an open neighborhood $U_l\subset R^n$ of $\bvec{x}^{(0)}$ such that $x_{l,1},\dots,x_{l,n}$ and $b_l$ are continuous on $U_l$, and $x_{l,l}$ is non-zero on $U_l$. Once we construct the $U_l$'s, $U:=U_1\cap\dots\cap U_n$ satisfies the desired condition. First, note that $x_{l,l}(\bvec{x}^{(0)})\neq 0$ for all $l$ by assumption. The assertion is trivial for $l=1$ by setting $U_1:=\{\bvec{x}\in R^n;x_1(\bvec{x})\neq 0\}$. Given $U_{l-1}$, let $U'_l:=U_{l-1}\cap\{\bvec{x}\in R^n;x_{l-1,l-1}(\bvec{x})\neq 0\}$, which is an open neighborhood of $\bvec{x}^{(0)}$. By the induction hypothesis, $x_{l,i}=\sigma(x_{l-1,i})q^{\s{i}}-b_{l-1} x_{l-1,i}$ is continuous on $U'_l$. We set $U_l:=U'_l\cap\{\bvec{x}\in R^n;x_{l,l}(\bvec{x})\neq 0\}$. Then, $U_l\subset R^n$ is an open neighborhood of $\bvec{x}^{(0)}$ on which $b_l$ is continuous on $U_l$ as desired.
\end{enumerate}
\end{proof}

\begin{lem}\label{lem:calc}
Let $\s{1}\le \s{2}\le \dots \le \s{n}$ be rational numbers such that $q^{\s{i}}\in\mathbb{Q}$. Then, the slopes of $f(\sigma):=(\sigma-q^{\s{1}}x)\dots (\sigma-q^{\s{n}}x)\in \Gamma[p^{-1}]\{\sigma\}$ are $-\s{n}\le\dots\le -\s{1}$ with multiplicities. Moreover, $f(\sigma)$ satisfies the condition~(*).
\end{lem}
\begin{proof}
We write
\[
f(\sigma)=\sigma^n+a_{n-1}\sigma^{n-1}+\dots+a_0,\ a_i\in \Gamma[p^{-1}].
\]
Then, we have
\[
a_{n-i}=\sum_{1\le j(1)<\dots<j(i)\le n}(-1)^iq^{\s{j(1)}+\dots+\s{j(i)}}x^{q^{j(1)-1}+\dots+q^{j(i)-1}}.
\]
Hence, $-\log_q{\ab{a_{n-i}}{1}}=\s{1}+\dots+\s{i}$, which implies the assertion.
\end{proof}

\begin{proof}[{Proof of Theorem~\ref{thm:cyc}}]

We first consider the case where $M$ is a $\sigma$-module over $R:=\Gamma^{\alg}[p^{-1}]$. Let $\s{1}\le \s{2}\le \dots\le \s{n}$ be the Frobenius slopes of $M$ with multiplicities. By Lemma~\ref{lem:calc}, the slopes of the $\sigma$-module $M':=R\{\sigma\}/R\{\sigma\}(\sigma-q^{\s{1}}x)(\sigma-q^{\s{2}}x)\dots (\sigma-q^{\s{n}}x)$ are $\s{1},\dots,\s{n}$ with multiplicities. Recall that $\sigma$-modules over $R$ are classified by its slopes with multiplicities by Dieudonn\'e-Manin theorem (\cite[14.6.3]{pde}). Hence, there exists an isomorphism of $\sigma$-modules $M\cong M'$. By Lemma~\ref{lem:calc}, $\bar{1}\in M'$ is a generic cyclic vector.

We consider the case where $M$ is a $\sigma$-module over $\Gamma^{\alg}_{\con}[p^{-1}]$. We have only to prove that there exists $f\in M$ which is a generic cyclic vector of $M^{\alg}:=\Gamma^{\alg}[p^{-1}]\otimes_{\Gamma^{\alg}_{\con}[p^{-1}]}M$. We apply Construction~\ref{const:Ked} to $M^{\alg}$. We choose a generic cyclic vector $e$ of $M^{\alg}$ and write $e=v(\bvec{x}^{(0)})$ with $\bvec{x}^{(0)}\in R^n$. By Lemma~\ref{lem:conti}~(ii) and (iii), there exists an open neighborhood $U\subset R^n$ of $\bvec{x}^{(0)}$ such that $v(\bvec{x})$ is a generic cyclic vector of $M^{\alg}$ for all $\bvec{x}\in U$. We choose a $\Gamma^{\alg}_{\con}[p^{-1}]$-basis $f_1,\dots,f_n$ of $M$. For $\bvec{y}=(y_1,\dots,y_n)\in R^n$, we define $w(\bvec{y}):=y_1f_1+\dots+y_nf_n$. For $\bvec{x}\in R^n$, there exists a unique $\bvec{y}=\bvec{y}(\bvec{x})\in R^n$ such that $v(\bvec{x})=w(\bvec{y})$, and the map $\bvec{x}\mapsto \bvec{y}(\bvec{x})$ is a homeomorphism (\cite[1.3.3]{pde}). Hence, there exists an open neighborhood $V\subset R^n$ of $\bvec{y}(\bvec{x}^{(0)})$ such that $w(\bvec{y})$ is a generic cyclic vector of $M^{\alg}$ for all $\bvec{y}\in V$. Since $\Gamma^{\alg}_{\con}$ is dense in $\Gamma^{\alg}$, $w(\bvec{y})\in M$ for $\bvec{y}\in V\cap (\Gamma^{\alg}_{\con}[p^{-1}])^n\neq \phi$ is a generic cyclic vector of $M^{\alg}$.
\end{proof}

\section{Frobenius equation and log-growth}\label{sec:Frob}

In \cite[\S~7.2]{CT}, Chiarellotto and Tsuzuki compute the log-growth of a solution $y$ of a Frobenius equation
\[
ay+by^{\sigma}+cy^{\sigma^2}=0,\ a,b,c\in K[\![x]\!]_0.
\]
In \cite{Nak}, Nakagawa proves a generalization of Chiarellotto-Tsuzuki's result for a Frobenius equation
\[
a_0y+a_1y^{\sigma}+\dots+a_ny^{\sigma^n}=0,\ a_i\in\mathcal{E}^{\dagger}
\]
under the assumption that the number of breaks of the Newton polygon of $a_0+a_1\sigma+\dots+a_n\sigma^n$ is equal to $n$. We generalize Nakagawa's result without any assumption on the Newton polygon:

\begin{thm}[{A generalization of Nakagawa's theorem (\cite[1.1]{Nak})}]\label{thm:Frob}
Let
\[
f(\sigma)=a_0+a_1\sigma+\dots+a_n\sigma^n\in \Gamma_{\con}^{\alg}[p^{-1}]\{\sigma\},\ a_0\neq 0,\ a_n\neq 0,\ n\ge 1
\]
be a twisted polynomial satisfying the condition (*) in Definition~\ref{dfn:cyc}~(ii) with slopes $-\s{1}<\dots<-\s{k}$. If $y\in\Gamma_{\log,\an,\con}^{\alg}$ is a solution of the Frobenius equation
\begin{equation}\label{eq:Frob}
f(\sigma)y=a_0y+a_1y^{\sigma}+\dots+a_ny^{\sigma^n}=0,
\end{equation}
then $y$ is either bounded or exactly of log-growth $\s{j}$ for some $j$ such that $\s{j}>0$.
\end{thm}

\begin{rem}\label{rem:PBQ}
For a $(\sigma,\nabla)$-module $M$ of rank $n$ over $\Gamma_{\con}[p^{-1}]$, we construct a Frobenius equation $f(\sigma)y=0$ satisfying the assumption of Theorem~\ref{thm:Frob} (see Construction~\ref{const:CT}). The ambiguity of the log-growth of $y$ in Theorem~\ref{thm:Frob} is owing to the fact that $M_{\mathcal{E}}/M_{\mathcal{E}}^0$ may not be pure as a $\sigma$-module. One can expect that if $M$ is PBQ, then $y$ is exactly of log-growth $s_1$, as is the case for $n=2$ (\cite[7.3]{CT}).
\end{rem}

We divide the proof into two parts: the first part is an estimation of an upper bound of the log-growth of $y$ (easier), and the second part is an estimation of a lower bound of the log-growth of $y$ (harder). The condition (*) will be used only in the second part. The integer $j$ in Theorem~\ref{thm:Frob} will be determined in \S~\ref{subsec:concl}.

\begin{notation}\label{not:rho}
In this section, we keep the notation in Theorem~\ref{thm:Frob}. Let $0=\bk{0}< \bk{1}< \dots< \bk{k}=n$ be the $x$-coordinates of the vertices of $\NP(f(\sigma))$. By Lemma~\ref{lem:lin}, there exists a real number $\rho_0$ sufficiently close to $1$ from the left such that for all $i\in \{0,1,\dots,n\}$, we have
\[
a_i\in\Gamma^{\alg}_{r},\ y\in\Gamma^{\alg}_{\log,\an,r},
\]
where $r=-\log_p{\rho_0^{q^n}}$ and
\[
\ab{a_i}{\rho}=\rho^{\alpha(i)}\ab{a_i}{1}\ \forall\rho\in [\rho_0^{q^n},1)
\]
for some $\alpha(i)\in\mathbb{Q}$; we fix such a $\rho_0$.
\end{notation}

\subsection{Estimation of upper bound}

\begin{prop}[{A refinement of \cite[6.12]{CT}}]\label{prop:upper}
Let $j\in \{0,\dots,k-1\}$. We assume
\begin{equation}\label{eq:ass}
\sup_{\bk{j-1}\le i\le \bk{j}}\ab{a_iy^{\sigma^i}}{\rho}\le \sup_{\bk{j}+1\le i\le n}\ab{a_iy^{\sigma^i}}{\rho}\ \forall\rho\in [\rho_0,1);
\end{equation}
when $j=0$, we set $\sup_{\bk{j-1}\le i\le \bk{j}}\ab{a_iy^{\sigma^i}}{\rho}=\ab{a_0y}{\rho}$. Then, we have
\begin{enumerate}
\item For any $\rho\in [\rho_0,1)$ and $m\in\mathbb{N}$, there exist an integer $N\in\{0,\dots,n-1\}$, which depends only on $m$, and a sequence $\varepsilon_{iu}$ of integers, which depends on $\rho$ and $m$, defined for
\[
I_{m}:=\{(i,u)\in \mathbb{Z}^2;\bk{j}+1\le i\le n,\ -m-\bk{j}\le u\le 0\}
\]
satisfying the following conditions:
\begin{enumerate}
\item[(a)]
\[
\log{\ab{y}{\rho^{q^{-m}}}}-\log{\ab{y}{\rho^{q^{-N}}}}\le \sum_{(i,u)\in I_{m}}\varepsilon_{iu}\log{\ab{a_i/a_{\bk{j}}}{\rho^{q^u}}};
\]
\item[(b)] $\varepsilon_{iu}\in\{0,1\}$ and
\[
\sum_{(i,u)\in I_{m}}(i-\bk{j})\varepsilon_{iu}=m-N.
\]
\end{enumerate}
\item $y$ has log-growth $\s{j+1}$.
\end{enumerate}
\end{prop}

\begin{proof}
\begin{enumerate}
\item We fix $\rho$ and proceed by induction on $m$. When $m\le n-1$, we set $N=m$ and $\varepsilon_{iu}\equiv 0$ for all $(i,u)\in I_m$. Then, we have nothing to prove. Assume that the assertion is true for the integers less than or equal to $m-1$ with $m\ge n$. By (\ref{eq:ass}) for $\rho=\rho^{q^{-m-\bk{j}}}$, we have
\begin{equation}\label{eq:1}
\ab{a_{\bk{j}}y^{\sigma^{\bk{j}}}}{\rho^{q^{-m-\bk{j}}}}\le \sup_{\bk{j}+1\le i\le n}\ab{a_iy^{\sigma^i}}{\rho^{q^{-m-\bk{j}}}}.
\end{equation}
We choose $i'\in\{\bk{j}+1,\dots,n\}$ such that
\begin{equation}\label{eq:2}
\sup_{\bk{j}+1\le i\le n}\ab{a_iy^{\sigma^i}}{\rho^{q^{-m-\bk{j}}}}=\ab{a_{i'}y^{\sigma^{i'}}}{\rho^{q^{-m-\bk{j}}}}.
\end{equation}
Recall $\ab{\sigma(\cdot)}{\eta}=\ab{\cdot}{\eta^q}$ for all $\eta\in (0,1)$ (\S~\ref{subsec:oc}). Then, by (\ref{eq:1}) and (\ref{eq:2}),
\begin{equation}\label{eq:3}
\log{\ab{y}{\rho^{q^{-m}}}}-\log{\ab{y}{\rho^{q^{-m-\bk{j}+i'}}}}\le \log{\ab{a_{i'}/a_{\bk{j}}}{\rho^{q^{-m-\bk{j}}}}}.
\end{equation}
By the induction hypothesis for $m+\bk{j}-i'$, there exist an integer $N\in\{0,\dots,n-1\}$ and a sequence $\varepsilon'_{iu}$ of $0$ or $1$ defined for $I_{m+\bk{j}-i'}$ such that
\begin{equation}\label{eq:4}
\log{\ab{y}{\rho^{q^{-m-\bk{j}+i'}}}}-\log{\ab{y}{\rho^{q^{-N}}}}\le \sum_{(i,u)\in I_{m+\bk{j}-i'}}\varepsilon'_{iu}\log{\ab{a_i/a_{\bk{j}}}{\rho^{q^{u}}}},
\end{equation}
\begin{equation}\label{eq:5}
\sum_{(i,u)\in I_{m+\bk{j}-i'}}(i-\bk{j})\varepsilon'_{iu}=m+\bk{j}-i'-N.
\end{equation}
For $(i,u)\in I_{m}$, we define
\[
\varepsilon_{iu}:=
\begin{cases}
\varepsilon'_{iu}&\text{if }(i,u)\in I_{m+\bk{j}-i'}\\
1&\text{if }(i,u)=(i',-m-\bk{j})\\
0&\text{otherwise}.
\end{cases}
\]
Then, by adding (\ref{eq:3}) to (\ref{eq:4}), the inequality in (a) follows. The condition $\varepsilon_{iu}\in \{0,1\}$ follows by construction, and the equality in (b) follows from (\ref{eq:5}).
\item We fix $\rho\in [\rho_0,1)$ for a while. Let $m$ be a natural number such that $\rho^{q^m}\in [\rho_0,\rho_0^{q^{-1}})$. By applying (i) to $(\rho,m)=(\rho^{q^m},m)$, there exist an integer $\NN{m}\in\{0,\dots,n-1\}$ and a sequence $\varepsilon_{iu}^{(m)}$ of $0$ or $1$ defined for $(i,u)\in I_{m}$ such that
\begin{equation}\label{eq:2.3}
\log{\ab{y}{\rho}}-\log{\ab{y}{\rho^{q^{m-\NN{m}}}}}\le \sum_{(i,u)\in I_{m}}\varepsilon^{(m)}_{iu}\log{\ab{a_i/a_{\bk{j}}}{\rho^{q^{u+m}}}}
\end{equation}
\begin{equation}\label{eq:2.4}
\sum_{(i,u)\in I_{m}}(i-\bk{j})\varepsilon^{(m)}_{iu}=m-\NN{m}.
\end{equation}
For $\bk{j}+1\le i\le n$, there exists $\x{i}\in\mathbb{Q}$ such that
\begin{equation}\label{eq:2.1}
\ab{a_i/a_{\bk{j}}}{\eta}=\eta^{\x{i}}\ab{a_i/a_{\bk{j}}}{1}\ \forall\eta\in [\rho_0,1)
\end{equation}
by Notation~\ref{not:rho}. Moreover,
\begin{equation}\label{eq:2.2}
-\frac{1}{i-\bk{j}}\log_q{\ab{a_i/a_{\bk{j}}}{1}}\ge -\s{j+1}
\end{equation}
by the convexity of the Newton polygon of $f(\sigma)$. By (\ref{eq:2.1}) and (\ref{eq:2.2}),
\begin{equation}\label{eq:2.5}
\text{RHS of }(\ref{eq:2.3})\le \sum_{(i,u)\in I_{m}}\varepsilon^{(m)}_{iu} q^{u+m} \x{i}\log{\rho}+\sum_{(i,u)\in I_{m}}(i-\bk{j})\varepsilon^{(m)}_{iu}\s{j+1}\log{q}.
\end{equation}

Let $\xx:=\max\{\pm \x{i};\bk{j}+1\le i\le n\}$. Then, the first sum in RHS of (\ref{eq:2.5}) is bounded above by
\begin{align*}
&\sum_{(i,u)\in I_{m}}\varepsilon^{(m)}_{iu}q^{u+m}\xx \log{(1/\rho)}\le\sum_{(i,u)\in I_{m}}q^{u+m} \xx \log{(1/\rho)}=(n-\bk{j})\frac{q^m-q^{-\bk{j}-1}}{1-q^{-1}}\xx \log{(1/\rho)}\\
\le& (n-\bk{j})\frac{q^{m+1}}{q-1}\xx \log{(1/\rho)}=(n-\bk{j})\frac{q}{q-1}\xx \log{(1/\rho^{q^m})} \le n \frac{q}{q-1}\xx\log{(1/\rho_0)}.
\end{align*}

By (\ref{eq:2.4}), the second sum in RHS of (\ref{eq:2.5}) is equal to
\[
(m-\NN{m})\s{j+1}\log{q}.
\]
Thus, (\ref{eq:2.3}) leads to
\begin{align}\label{eq:2.6}
\ab{y}{\rho}&\le C \ab{y}{\rho^{q^{m-\NN{m}}}}\cdot q^{(m-\NN{m}) \s{j+1}}\notag\\&=C \ab{y}{\rho^{q^{m-\NN{m}}}}\cdot (\log{(1/\rho^{q^{m-\NN{m}}})})^{\s{j+1}}\cdot\Ro{-\s{j+1}},
\end{align}
where $C:=\exp\{n q(q-1)^{-1} \xx\log{(1/\rho_0)}\}$ is a constant independent of $\rho$. Since $\rho^{q^{m-\NN{m}}}\in [\rho_0,\rho_0^{q^{-n}})$, the functions $\ab{y}{\rho^{q^{m-\NN{m}}}}$ and $(\log{(1/\rho^{q^{m-\NN{m}}})})^{\s{j+1}}$ are bounded when $\rho$ runs over $[\rho_0,1)$: note that the function $[\rho_0,1)\to\mathbb{R};\rho\mapsto \ab{y}{\rho}$ is continuous. Thus, (\ref{eq:2.6}) implies the desired estimation
\[
\ab{y}{\rho}=\RO{-\s{j+1}}\text{ as }\rho\uparrow 1.
\]
\end{enumerate}
\end{proof}

\subsection{Estimation of lower bound}

We start with converting the condition (*) into the lemma:

\begin{lem}\label{lem:seg}
For any $j\in\{0,\dots,k-1\}$, $i\in \{\bk{j+1}+1,\dots,n\}$, and $i'\in\{0,\dots,\bk{j+1}-1\}$, we have
\[
\log{\ab{a_{i'-\bk{j+1}+i}}{1}}-\log{\ab{a_{i'}}{1}}>\log{\ab{a_i}{1}}-\log{\ab{a_{\bk{j+1}}}{1}}.
\]
\end{lem}
\begin{proof}
For $0\le i\le n$, we denote by $P_i$ the point $(i,-\log_q{\ab{a_i}{1}})$. We also denote by $L_1$ and $L_2$ the segments $P_{i'}P_{i'-\bk{j+1}+i}$ and $P_{\bk{j+1}}P_i$, respectively. Let $a$ and $b$ be the slopes of $L_1$ and $L_2$, respectively. We have only to prove $a<b$. Let us consider separately the cases where $i'-\bk{j+1}+i\le \bk{j+1}$ or $i'-\bk{j+1}+i>\bk{j+1}$. In the first case, we have $a\le-\s{j+1}$ by the condition~(*). By the convexity of the Newton polygon of $f(\sigma)$, we have $-\s{j+1}< b$. Hence, $a<b$. In the latter case, the segment $L_1$ intersects with $L_2$. Since $P_{\bk{j+1}}$ is under $L_1$, we have $a<b$.
\end{proof}

\begin{notation}\label{not:rho2}
By Lemmas~\ref{lem:lin} and \ref{lem:seg}, after choosing $\rho_0$ sufficiently large if necessary, we may assume the following condition: for any $j\in\{0,\dots,k-1\}$, $i\in \{\bk{j+1}+1,\dots,n\}$, and $i'\in\{0,\dots,\bk{j+1}-1\}$, we have
\begin{equation}\label{eq:key}
\log{\ab{a_{i'-\bk{j+1}+i}}{\rho_2}}-\log{\ab{a_{i'}}{\rho_2}}>\log{\ab{a_i}{\rho_3}}-\log{\ab{a_{\bk{j+1}}}{\rho_3}}\ \forall\rho_2,\rho_3\in [\rho_0,1);
\end{equation}
indeed, both sides of the inequality are continuous with respect to $\rho_2$ and $\rho_3$, respectively, and converge to $\log{\ab{a_{i'-\bk{j+1}+i}}{1}}-\log{\ab{a_{i'}}{1}}$ and $\log{\ab{a_i}{1}}-\log{\ab{a_{\bk{j+1}}}{1}}$ as $\rho_2,\rho_3\uparrow 1$, respectively.
\end{notation}

To estimate the function $\ab{y}{\rho}$ of $\rho$ from below, we need to combine several inequalities which are similar to inequality (\ref{eq:ass}).

\begin{assumption}\label{ass:ineq}
Let $j\in \{0,\dots,k-1\}$. In the rest of this subsection, we assume the following:
\[
\sup_{\bk{0}\le i\le \bk{1}}{\ab{a_iy^{\sigma^i}}{\rho}}\le \sup_{\bk{1}+1\le i\le n}{\ab{a_iy^{\sigma^i}}{\rho}}\ \forall\rho\in [\rho_0^{q^{-n}},1),
\]
\[
\sup_{\bk{1}\le i\le \bk{2}}{\ab{a_iy^{\sigma^i}}{\rho}}\le \sup_{\bk{2}+1\le i\le n}{\ab{a_iy^{\sigma^i}}{\rho}}\ \forall\rho\in [\rho_0^{q^{-n}},1),
\]
\[
\vdots
\]
\[
\sup_{\bk{j-1}\le i\le \bk{j}}{\ab{a_iy^{\sigma^i}}{\rho}}\le \sup_{\bk{j}+1\le i\le n}{\ab{a_iy^{\sigma^i}}{\rho}}\ \forall\rho\in [\rho_0^{q^{-n}},1);
\]
when $j=0$, we set $\sup_{\bk{j-1}\le i\le \bk{j}}{\ab{a_iy^{\sigma^i}}{\rho}}:=\ab{a_0y}{\rho}$. We also assume
\[
\sup_{\bk{j}\le i\le \bk{j+1}}{\ab{a_iy^{\sigma^i}}{\rho}}>\sup_{\bk{j+1}+1\le i\le n}{\ab{a_iy^{\sigma^i}}{\rho}}\ \exists \rho\in [\rho_0^{q^{-n}},1);
\]
when $j=k-1$, we set $\sup_{\bk{j+1}+1\le i\le n}{\ab{a_iy^{\sigma^i}}{\rho}}:=0$.
\end{assumption}

\begin{lem}\label{lem:lower}
Assume that $\rho_1\in [\rho_0^{q^{-n}},1)$ satisfies
\[
\sup_{\bk{j}\le i\le \bk{j+1}}{\ab{a_iy^{\sigma^i}}{\rho_1}}>\sup_{\bk{j+1}+1\le i\le n}{\ab{a_iy^{\sigma^i}}{\rho_1}}.
\]
\begin{enumerate}
\item We have
\[
\sup_{0\le i\le \bk{j+1}-1}\ab{a_iy^{\sigma^i}}{\rho_1}\ge \ab{a_{\bk{j+1}}y^{\sigma^{\bk{j+1}}}}{\rho_1}.
\]
\item Let $i'\in\{0,\dots,\bk{j+1}-1\}$ be an integer such that
\[
\ab{a_{i'}y^{\sigma^{i'}}}{\rho_1}=\sup_{0\le i\le \bk{j+1}-1}{\ab{a_iy^{\sigma^i}}{\rho_1}}.
\]
Then, we have
\[
\sup_{\bk{j}\le i\le \bk{j+1}}{\ab{a_iy^{\sigma^{i}}}{\rho_1^{q^{i'-\bk{j+1}}}}}>\sup_{\bk{j+1}+1\le i\le n}{\ab{a_iy^{\sigma^i}}{\rho_1^{q^{i'-\bk{j+1}}}}}.
\]
\end{enumerate}
\end{lem}
\begin{proof}
\begin{enumerate}
\item Suppose the contrary. Then, we have $\ab{a_{\bk{j+1}}y^{\sigma^{\bk{j+1}}}}{\rho_1}>\sup_{i\neq \bk{j+1}}{\ab{a_iy^{\sigma^i}}{\rho_1}}\ge 0$ by the assumption of the lemma. By (\ref{eq:Frob}), we have $a_{\bk{j+1}}y^{\sigma^{\bk{j+1}}}=-\sum_{i\neq\bk{j+1}}a_iy^{\sigma^i}$. By taking $\ab{\cdot}{\rho_1}$, we have $\ab{a_{\bk{j+1}}y^{\sigma^{\bk{j+1}}}}{\rho_1}\le\sup_{i\neq\bk{j+1}}\ab{a_iy^{\sigma^i}}{\rho_1}$, which is a contradiction.
\item By (i) and the assumption of the lemma, we have
\begin{equation}\label{eq:3.1}
\ab{a_{i'}y^{\sigma^{i'}}}{\rho_1}=\sup_{0\le i\le n}\ab{a_{i}y^{\sigma^{i}}}{\rho_1}.
\end{equation}
For $i\in\{\bk{j+1}+1,\dots,n\}$, we have
\begin{align*}
&\log{\ab{y^{\sigma^{\bk{j+1}}}}{\rho_1^{q^{i'-\bk{j+1}}}}}-\log{\ab{y^{\sigma^i}}{\rho_1^{q^{i'-\bk{j+1}}}}}=\log{\ab{y}{\rho_1^{q^{i'}}}}-\log{\ab{y}{\rho_1^{q^{i'-\bk{j+1}+i}}}}\\
=&\log{\ab{a_{i'}y^{\sigma^{i'}}}{\rho_1}}-\log{\ab{a_{i'-\bk{j+1}+i}y^{\sigma^{i'-\bk{j+1}+i}}}{\rho_1}}-\log{\ab{a_{i'}}{\rho_1}}+\log{\ab{a_{i'-\bk{j+1}+i}}{\rho_1}}\\
\ge&\log{\ab{a_{i'-\bk{j+1}+i}}{\rho_1}}-\log{\ab{a_{i'}}{\rho_1}}>\log{\ab{a_i}{\rho_1^{q^{i'-\bk{j+1}}}}}-\log{\ab{a_{\bk{j+1}}}{\rho_1^{q^{i'-\bk{j+1}}}}},
\end{align*}
where the first and second inequalities follow from (\ref{eq:3.1}) and (\ref{eq:key}), respectively. Thus, we obtain
\[
\ab{a_{\bk{j+1}}y^{\sigma^{\bk{j+1}}}}{\rho_1^{q^{i'-\bk{j+1}}}}>\sup_{\bk{j+1}+1\le i\le n}\ab{a_iy^{\sigma^i}}{\rho_1^{q^{i'-\bk{j+1}}}},
\]
which implies the assertion.
\end{enumerate}
\end{proof}


\begin{construction}\label{const:lower}
Fix $\rho_1\in [\rho_0^{q^{-n}},1)$ such that
\[
\sup_{\bk{j}\le i\le \bk{j+1}}{\ab{a_iy^{\sigma^i}}{\rho_1}}>\sup_{\bk{j+1}+1\le i\le n}{\ab{a_iy^{\sigma^i}}{\rho_1}}.
\]
By induction on $l\in\mathbb{N}$, we construct a strictly decreasing sequence $\{\m{l}\}_l$ of integers less than or equal to $\bk{j+1}$, and a sequence $\varepsilon_{iu}^{(l)}$ of integers defined for
\[
\mathcal{I}_{l}:=\{(i,u)\in\mathbb{Z}^2;0\le i\le  \bk{j+1}-1,\m{l}-\bk{j+1}\le u\le 0\}
\]
satisfying the following conditions:
\begin{enumerate}
\item[(a)]
\[
\log{\ab{y}{\rho_1^{q^{\m{l}}}}}-\log{\ab{y}{\rho_1^{q^{\bk{j+1}}}}}\ge\sum_{(i,u)\in\mathcal{I}_{l}}\varepsilon_{iu}^{(l)}\log{\ab{a_{\bk{j+1}}/a_i}{\rho_1^{q^u}}}.
\]
\item[(b)] $\varepsilon_{iu}^{(l)}\in\{0,1\}$ and
\[
\sum_{(i,u)\in\mathcal{I}_{l}}(\bk{j+1}-i)\varepsilon_{iu}^{(l)}=\bk{j+1}-\m{l}.
\]
\item[(c)]
\[
\sup_{\bk{j}\le i\le \bk{j+1}}{\ab{a_iy^{\sigma^i}}{\rho_1^{q^{\m{l}-\bk{j+1}}}}}>\sup_{\bk{j+1}+1\le i\le n}{\ab{a_iy^{\sigma^i}}{\rho_1^{q^{\m{l}-\bk{j+1}}}}}.
\]
\end{enumerate}

We set $\m{0}:=i'$ where $i'$ is defined in Lemma~\ref{lem:lower}~(ii), and define
\[
\varepsilon_{iu}^{(0)}:=
\begin{cases}
1&\text{if }(i,u)=(i',0)\\
0&\text{otherwise}.
\end{cases}
\]
Since $\ab{a_{i'}y^{\sigma^{i'}}}{\rho_1}\ge \ab{a_{\bk{j+1}}y^{\sigma^{\bk{j+1}}}}{\rho_1}$ by Lemma~\ref{lem:lower}~(i), condition (a) follows. Condition (b) follows by definition. Condition (c) follows from Lemma~\ref{lem:lower}~(ii).

Given $\m{l}$ and $\varepsilon_{iu}^{(l)}$, we can apply Lemma~\ref{lem:lower} to $\rho_1=\rho_1^{q^{\m{l}-\bk{j+1}}}$ by condition (c) for $\m{l}$: let $i'\in\{0,\dots,\bk{j+1}-1\}$ be the integer defined in Lemma~\ref{lem:lower}~(ii). Since $\ab{a_{i'}y^{\sigma^{i'}}}{\rho_1^{q^{\m{l}-\bk{j+1}}}}\ge \ab{a_{\bk{j+1}}y^{\sigma^{\bk{j+1}}}}{\rho_1^{q^{\m{l}-\bk{j+1}}}}$ by Lemma~\ref{lem:lower}~(i), we have
\begin{equation}\label{eq:3.3}
\log{\ab{y}{\rho_1^{q^{\m{l}-\bk{j+1}+i'}}}}-\log{\ab{y}{\rho_1^{q^{\m{l}}}}}\ge \log{\ab{a_{\bk{j+1}}/a_{i'}}{\rho_1^{q^{\m{l}-\bk{j+1}}}}}.
\end{equation}
We set $\m{l+1}:=\m{l}-\bk{j+1}+i'<\m{l}$ and define $\varepsilon_{iu}^{(l+1)}$ for $(i,u)\in\mathcal{I}_{l+1}$ by
\[
\varepsilon_{iu}^{(l+1)}:=
\begin{cases}
\varepsilon_{iu}^{(l)}&\text{if }(i,u)\in\mathcal{I}_{l}\\
1&\text{if }(i,u)=(i',\m{l+1}-\bk{j+1})\\
0&\text{otherwise.}
\end{cases}
\]
We verify the conditions (a), (b), and (c). By adding (\ref{eq:3.3}) to the inequality in (a) for $\m{l}$, condition (a) follows. Condition (b) follows by construction. Condition (c) follows from Lemma~\ref{lem:lower}~(ii).
\end{construction}

\begin{prop}\label{prop:lower}
If $y$ is non-zero and has log-growth $\alpha\in\mathbb{R}_{\ge 0}$, then $\alpha\ge \s{j+1}$.
\end{prop}
\begin{proof}
Obviously, we may assume $\s{j+1}>0$. For $i\in\{0,\dots,\bk{j+1}-1\}$, there exists $\x{i}\in\mathbb{Q}$ such that
\begin{equation}\label{eq:4.1}
\ab{a_{\bk{j+1}}/a_i}{\rho}=\rho^{\x{i}}\ab{a_{\bk{j+1}}/a_i}{1}\ \forall\rho\in [\rho_0^{q^n},1).
\end{equation}
Moreover,
\begin{equation}\label{eq:4.2}
-\frac{1}{\bk{j+1}-i}\log_q{\ab{a_{\bk{j+1}}/a_i}{1}}\le -\s{j+1}
\end{equation}
by the convexity of the Newton polygon of $f(\sigma)$.

We retain the notation in Construction~\ref{const:lower}. By (\ref{eq:4.1}), (\ref{eq:4.2}), and the inequality (a) for $\m{l}$, we have
\begin{equation}\label{eq:4.3}
\log{\ab{y}{\rho_1^{q^{\m{l}}}}}-\log{\ab{y}{\rho_1^{q^{\bk{j+1}}}}}\ge\sum_{(i,u)\in\mathcal{I}_{l}}\varepsilon^{(l)}_{iu}q^u\x{i}\log{\rho_1}+\sum_{(i,u)\in\mathcal{I}_{l}}(\bk{j+1}-i)\varepsilon^{(l)}_{iu}\s{j+1}\log{q}.
\end{equation}
Let $\xx:=\inf\{\pm \x{i};0\le i\le \bk{j+1}-1\}$. Then, the first sum in RHS of (\ref{eq:4.3}) is bounded below by
\[
\sum_{(i,u)\in\mathcal{I}_{l}}\varepsilon_{iu}^{(l)}q^u\xx\log{(1/\rho_1)}\ge\sum_{(i,u)\in\mathcal{I}_{l}}q^u\xx\log{(1/\rho_1)}= \bk{j+1} \frac{1-q^{\m{l}-\bk{j+1}-1}}{1-q^{-1}}  \xx\log{(1/\rho_1)}\ge n\frac{q}{q-1}\xx\log{(1/\rho_1)}.
\]
By condition (b) for $\m{l}$, the second sum in RHS of (\ref{eq:4.3}) is equal to
\[
(\bk{j+1}-\m{l})\s{j+1}\log{q}.
\]
Therefore, (\ref{eq:4.3}) leads to
\begin{equation}\label{eq:4.6}
\ab{y}{\rho_1^{q^{\m{l}}}}\ge C \ab{y}{\rho_1^{q^{\bk{j+1}}}}\cdot q^{(\bk{j+1}-\m{l})\s{j+1}}=C\ab{y}{\rho_1^{q^{\bk{j+1}}}}\cdot q^{\bk{j+1}\s{j+1}}(\log{(1/\rho_1)})^{\s{j+1}}\cdot (\log{(1/\rho_1^{q^{\m{l}}})})^{-\s{j+1}},
\end{equation}
where $C:=\exp\{nq(q-1)^{-1} \xx \log{(1/\rho_1)}\}$. Note that
\[
C\ab{y}{\rho_1^{q^{\bk{j+1}}}}\cdot q^{\bk{j+1}\s{j+1}}(\log{(1/\rho_1)})^{\s{j+1}}
\]
is a positive constant independent of $l$. Since $\m{l}\to-\infty$ as $l\to\infty$, (\ref{eq:4.6}) implies
\[
\ab{y}{\rho}\neq \RO{-\beta}\text{ as }\rho\uparrow 1
\]
for any $\beta\in\mathbb{R}_{<\s{j+1}}$. In other words, $y$ does not have log-growth strictly less than $s_{j+1}$. Hence, $\alpha\ge s_{j+1}$.
\end{proof}


\subsection{Proof of Theorem~\ref{thm:Frob}}\label{subsec:concl}
Let $\rho_0$ be as in Notation~\ref{not:rho2}. For $j\in\{0,1,\dots,k-2\}$, we consider the following condition on $y$:
\[
(C_j):\sup_{\bk{j}\le i\le \bk{j+1}}\ab{a_iy^{\sigma^i}}{\rho}\le \sup_{\bk{j+1}+1\le i\le n}\ab{a_iy^{\sigma^i}}{\rho}\ \forall\rho\in [\rho_0,1).
\]
Let $j\in\{0,\dots,k-2\}$ be the least integer such that condition $(C_{j})$ does not hold; if condition $(C_j)$ holds for all $j$, then we set $j=k-1$. Then, $j$ satisfies the assumption in Proposition~\ref{prop:upper}; when $j=0$, the assumption follows from (\ref{eq:Frob}). In addition, $j$ satisfies Assumption~\ref{ass:ineq}; when $j=k-1$, the assumption follows from $y\neq 0$. Therefore, the assertion follows from Propositions~\ref{prop:upper} and \ref{prop:lower}.


\section{Proof of Theorem~\ref{thm:main2}}\label{sec:pf}

In this section, we assume that $k_K$ is algebraically closed as in Assumption~\ref{ass:alg}. For a $(\sigma,\nabla)$-module over $K[\![x]\!]_0$, Chiarellotto and Tsuzuki define a Frobenius equation (\cite[Proof of Theorem~6.17~(2)]{CT}). Then, they interpret their conjecture ${\bf LGF}_{K[\![x]\!]_0}$ as a problem on the Frobenius equation. For a $(\sigma,\nabla)$-module over $\Gamma_{\con}[p^{-1}]$, their method can be applied as follows.

\begin{construction}\label{const:CT}
Set $R:=\Gamma_{\con}[p^{-1}]$. Let $M$ be a $(\sigma,\nabla)$-module of rank $n$ over $R$ solvable in $\Gamma_{\log,\an,\con}$, and $\s{1}<\dots<\s{k}$ the generic Frobenius slopes of $M$. Assume $q^{\s{j}}\in\mathbb{Q}$ for all $j$. We choose a generic cyclic vector $e$ of $\Gamma^{\alg}_{\con}[p^{-1}]\otimes_R{M\spcheck}$ by Theorem~\ref{thm:cyc}, and we write
\[
\varphi^n(e)=-(a_{n-1}\varphi^{n-1}(e)+\dots+a_{0}e),\ a_i\in\Gamma^{\alg}_{\con}[p^{-1}].
\]
Let $v$ be an element of $\mathfrak{Sol}(M)$ such that $\varphi(v)=\gamma v$ for some $\gamma\in (\Gamma^{\alg}_{\con}[p^{-1}])^{\times}$. By identifying $\mathfrak{Sol}(M)$ as a submodule of $\Gamma_{\log,\an,\con}^{\alg}\otimes_RM\spcheck$, we write
\[
v=y_0e+y_1\varphi(e)+\dots+y_{n-1}\varphi^{n-1}(e),\ y_i\in\Gamma^{\alg}_{\log,\an,\con}.
\]
Then, we obtain the relation
\begin{equation}\label{eq:relation}
\begin{pmatrix}
&&&-a_0\\
1&&&-a_{1}\\
&\ddots&&\vdots\\
&&1&-a_{n-1}
\end{pmatrix}
\sigma
\begin{pmatrix}
y_0\\
y_1\\
\vdots\\
y_{n-1}
\end{pmatrix}
=\gamma
\begin{pmatrix}
y_0\\
y_1\\
\vdots\\
y_{n-1}
\end{pmatrix}.
\end{equation}
By elimination, $y:=y_{n-1}$ satisfies the following Frobenius equation:
\begin{equation}
y=-\sum_{0\le i\le n-1}\frac{\sigma^i(a_{n-i-1})}{\gamma \sigma(\gamma)\dots \sigma^{i}(\gamma)}\sigma^{i+1}(y).
\end{equation}
Note that the slopes of the twisted polynomial
\[
1+\frac{a_{n-1}}{\gamma}\sigma+\dots+\frac{\sigma^{n-1}(a_{0})}{\gamma \sigma(\gamma)\cdots \sigma^{n-1}(\gamma)}\sigma^n
\]
are $-\s{k}-s<\dots<-\s{1}-s$, where $s=-\log_q{|\gamma|}$.
\end{construction}

\begin{lem}\label{lem:CT}
We retain the notation in Construction~\ref{const:CT}.
\begin{enumerate}
\item For $\lambda\in\mathbb{R}$, we have $v\in\mathfrak{Sol}_{\lambda}(M)$ if and only if $y\in \Fil_{\lambda}\Gamma^{\alg}_{\log,\an,\con}$.
\item We have either $v\in\mathfrak{Sol}_0(M)$ or $v\in\mathfrak{Sol}_{s+\s{j}}(M)\setminus\mathfrak{Sol}_{(s+\s{j})-}(M)$ for some $j$ such that $s+\s{j}>0$.
\end{enumerate}
\end{lem}
\begin{proof}
\begin{enumerate}
\item Since we have
\[
\mathfrak{Sol}_{\lambda}(M)\subset \Fil_{\lambda}\Gamma_{\log,\an,\con}\otimes_{R}M\spcheck\subset \Fil_{\lambda}\Gamma^{\alg}_{\log,\an,\con}\otimes_{R}M\spcheck\cong \Fil_{\lambda}\Gamma^{\alg}_{\log,\an,\con}\otimes_{\Gamma^{\alg}_{\con}[p^{-1}]}\Gamma^{\alg}_{\con}[p^{-1}]\otimes_{R}M\spcheck,
\]
$v\in\mathfrak{Sol}_{\lambda}(M)$ implies $y\in \Fil_{\lambda}\Gamma^{\alg}_{\log,\an,\con}$. Assume $y\in \Fil_{\lambda}\Gamma^{\alg}_{\log,\an,\con}$. By (\ref{eq:relation}) and Lemma~\ref{lem:logan}, we have $y_i\in \Fil_{\lambda}\Gamma^{\alg}_{\log,\an,\con}$ by decreasing induction on $i$. Hence, $v\in \Fil_{\lambda}\Gamma^{\alg}_{\log,\an,\con}\otimes_{R}M\spcheck$. Since $\Fil_{\lambda}\Gamma^{\alg}_{\log,\an,\con}\cap\Gamma_{\log,\an,\con}=\Fil_{\lambda}\Gamma_{\log,\an,\con}$ by definition, we have $v\in \Fil_{\lambda}\Gamma_{\log,\an,\con}\otimes_{R}M\spcheck$, i.e., $v\in\mathfrak{Sol}_{\lambda}(M)$.
\item The assertion follows from (i) and Theorem~\ref{thm:Frob}.
\end{enumerate}
\end{proof}

We deduce Proposition~\ref{prop:basic} and Theorem~\ref{thm:main2} from Lemma~\ref{lem:CT}~(ii) and the following lemma.

\begin{lem}\label{lem:PBQ}
Let $M$ be a $(\sigma,\nabla)$-module over $\mathcal{E}$ and $\lambda_{\max}$ the highest Frobenius slope of $M$. If $M$ is PBQ, then $(M\spcheck)_0$ is pure of slope $-\lambda_{\max}$ as a $\sigma$-module.
\end{lem}
\begin{proof}
We have a canonical isomorphism $(M\spcheck)_0\cong (M/M^0)\spcheck$ induced by the canonical paring $M\otimes_{\mathcal{E}}M\spcheck\to\mathcal{E}$ (\S~\ref{subsec:lgfil}). Hence, $(M\spcheck)_0$ is pure as a $\sigma$-module by assumption. Moreover, the Frobenius slope $\lambda$ of $(M\spcheck)_0$ is greater than or equal to $-\lambda_{\max}$. Suppose that the assertion is false, i.e., $\lambda>-\lambda_{\max}$. Let $M'$ be the inverse image of $M'':=S_{-\lambda_{\max}}(M\spcheck/(M\spcheck)_0)$ under the canonical projection $M\spcheck\to M\spcheck/(M\spcheck)_0$. By assumption, $M''\neq 0$ and there exists a short exact sequence of $(\sigma,\nabla)$-modules over $\mathcal{E}$:
\[\xymatrix{
0\ar[r]&(M\spcheck)_0\ar[r]&M'\ar[r]&M''\ar[r]&0.
}\]
By \cite[4.2]{CT2}, the above exact sequence splits as a sequence of $(\sigma,\nabla)$-modules. Since $M''_0\neq 0$, we have $(M\spcheck)_0\subsetneq M'_0\subset (M\spcheck)_0$, which is a contradiction.
\end{proof}

\begin{proof}[{Proof of Proposition~\ref{prop:basic}}]
By replacing $(M,\varphi,\nabla)$ by $(M,\varphi^h,\nabla)$ for sufficiently large $h\in\mathbb{N}$, we may assume $q^s\in\mathbb{Q}$ for all special and generic Frobenius slopes $s$ of $M$. Then, we may apply Construction~\ref{const:CT} to $M$. We retain the notation in Construction~\ref{const:CT}. Note that $\s{k}=\lambda_{\max}$ by definition.
\begin{enumerate}
\item Let $v\in \mathfrak{Sol}(M)$ be a non-zero Frobenius eigenvector of slope $s$. By Grothendieck-Katz specialization theorem (\cite[15.3.2]{pde}), we have $s\ge -\s{k}$. By Lemma~\ref{lem:CT}~(ii), we have $v\in \mathfrak{Sol}_{s+\s{k}}(M)$. Hence, we have $S_{\lambda-\lambda_{\max}}(\mathfrak{Sol}(M))\subset \mathfrak{Sol}_{\lambda}(M)$ for all $\lambda\in\mathbb{R}$. By taking $(\cdot)^{\perp}$ with respect to the canonical pairing $\mathfrak{V}(M)\otimes_K\mathfrak{Sol}(M)\to K$, we obtain $(S_{\lambda-\lambda_{\max}}(\mathfrak{V}(M\spcheck)))^{\perp}\supset \mathfrak{V}(M)^{\lambda}$.
\item By (i), we have only to prove $(S_{-\lambda_{\max}}(\mathfrak{V}(M\spcheck)))^{\perp}\subset \mathfrak{V}(M)^{0}$. Since $\mathfrak{Sol}_0(M)=(M\spcheck)^{\nabla=0}$, we have $\mathfrak{Sol}_0(M)\subset (M_{\mathcal{E}}\spcheck)_0$ by the characterization of $(M_{\mathcal{E}}\spcheck)_0$. By Lemma~\ref{lem:PBQ}, $(M_{\mathcal{E}}\spcheck)_0$, and hence, $\mathfrak{Sol}_0(M)$ are pure of slope $-\lambda_{\max}$ as a $\sigma$-module, i.e., $\mathfrak{Sol}_0(M)\subset S_{-\lambda_{\max}}(\mathfrak{Sol}(M))$. By taking $(\cdot)^{\perp}$ with respect to the canonical pairing $\mathfrak{V}(M)\otimes_K\mathfrak{Sol}(M)\to K$, we obtain the assertion.
\end{enumerate}
\end{proof}

\begin{proof}[{Proof of Theorem~\ref{thm:main2}}]
Similarly to the proof of Proposition~\ref{prop:basic}, we may apply Construction~\ref{const:CT} to $M$ again.
\begin{enumerate}
\item By the definition of $\mathfrak{V}(M)^{\bullet}$, we have only to prove that the filtration $\mathfrak{Sol}_{\bullet}(M)$ is rational and right continuous.

We first prove the rationality of breaks $\lambda$ of $\mathfrak{Sol}_{\bullet}(M)$. We may assume $\lambda>0$. Since $\mathfrak{Sol}_{\lambda-}(M)$ is a direct summand of $\mathfrak{Sol}_{\lambda +}(M)$ as a $\sigma$-module, we can choose a Frobenius eigenvector $v\in\mathfrak{Sol}_{\lambda +}(M)\setminus\mathfrak{Sol}_{\lambda -}(M)$ of slope $s$. By $v\notin\mathfrak{Sol}_0(M)$ and Lemma~\ref{lem:CT}~(ii), we have $v\in\mathfrak{Sol}_{s+\s{j}}(M)\setminus\mathfrak{Sol}_{(s+\s{j})-}(M)$ for some $j$ such that $s+\s{j}>0$, i.e., $\lambda=s+\s{j}\in\mathbb{Q}$.

We prove the right continuity of $\mathfrak{Sol}_{\bullet}(M)$. Suppose the contrary, i.e., there exists $\lambda\in\mathbb{R}_{\ge 0}$ such that $\mathfrak{Sol}_{\lambda}(M)\neq \mathfrak{Sol}_{\lambda +}(M)$. Let $\Delta(M)$ be the set of rational numbers consisting of $0$ and $s+\s{j}$ where $s$ is a Frobenius slope of $\mathfrak{Sol}(M)$. Fix $\lambda'\in \mathbb{R}_{>\lambda}$ sufficiently close to $\lambda$ such that $\mathfrak{Sol}_{\lambda +}(M)=\mathfrak{Sol}_{\lambda'}(M)$ and $\Delta(M)\cap(\lambda,\lambda']=\phi$. Since $\mathfrak{Sol}_{\lambda}(M)$ is a direct summand of $\mathfrak{Sol}_{\lambda +}(M)$ as a $\sigma$-module, we can choose a Frobenius eigenvector $v\in\mathfrak{Sol}_{\lambda +}(M)\setminus\mathfrak{Sol}_{\lambda}(M)$ of slope $s$. By Lemma~\ref{lem:CT}~(ii), we have either $v\in\mathfrak{Sol}_0(M)$ or $v\in\mathfrak{Sol}_{s+\s{j}}(M)\setminus\mathfrak{Sol}_{(s+\s{j})-}(M)$ for some $j$ such that $s+\s{j}>0$. In the first case, we have $v\in\mathfrak{Sol}_{\lambda}(M)$, which is a contradiction. In the latter case, we have $s+\s{j}>\lambda$ by $v\in\mathfrak{Sol}_{s+\s{j}}(M)$. Since $v\notin\mathfrak{Sol}_{(s+\s{j})-}(M)$, we have $\lambda'\ge s+\s{j}$. Hence, we have $s+\s{j}\in \Delta(M)\cap (\lambda,\lambda']=\phi$, which is a contradiction.

\item By Proposition~\ref{prop:basic}~(i), we have only to prove $\mathfrak{Sol}_{\lambda}(M)\subset S_{\lambda-\lambda_{\max}}(\mathfrak{Sol}(M))$ for all $\lambda\ge 0$. Let us consider separately the cases where $k=1$ or $k=2$. In the first case, since $\mathfrak{Sol}(M)$ is pure of slope $-\lambda_{\max}$ by Grothendieck-Katz specialization theorem, the assertion is trivial. In the latter case, let $v\in \mathfrak{Sol}_{\lambda}(M)$ be a non-zero Frobenius eigenvector of slope $s$. By Grothendieck-Katz specialization theorem, we have $-\s{2}\le s\le -\s{1}$. Hence, we have either $v\in \mathfrak{Sol}_0(M)$ or $v\in\mathfrak{Sol}_{s+\s{2}}(M)\setminus\mathfrak{Sol}_{(s+\s{2})-}(M)$ by Lemma~\ref{lem:CT}~(ii). In the first case, $v\in \mathfrak{Sol}_0(M)=S_{-\s{2}}(\mathfrak{Sol}(M))\subset S_{\lambda-\s{2}}(\mathfrak{Sol}(M))$ by Proposition~\ref{prop:basic}~(ii). In the latter case, we have $s+\s{2}\le \lambda$. Hence, $v\in S_s(\mathfrak{Sol}(M)) \subset S_{\lambda-\s{2}}(\mathfrak{Sol}(M))$.
\end{enumerate}
\end{proof}

Let $M$ be a $(\sigma,\nabla)$-module over $\Gamma_{\con}[p^{-1}]$ solvable in $\Gamma_{\log,\an,\con}$. We can expect that any break $\lambda$ of the special log-growth filtration of $M$ is of the form $-s+\lambda_{\max}$ where $s$ is a special Frobenius slope of $M$. At this point, as in the proof of Theorem~\ref{thm:main2}~(i), we can prove:

\begin{prop}
In the above setting, any break $\lambda$ of the special log-growth filtration of $M$ is of the form $-s+s'$ where $s$ (resp. $s'$) is a special (resp. generic) Frobenius slope of $M$ such that $-s+s'\ge 0$.
\end{prop}

\section{Appendix: diagram of rings}
For $0\le \lambda_1\le \lambda_2$, we have the following diagram of rings: all the morphisms are the natural inclusions.
\[\xymatrix{
&K[\![x]\!]_0\ar[r]\ar[d]&K[\![x]\!]_{\lambda_1}\ar[r]\ar[d]&K[\![x]\!]_{\lambda_2}\ar[r]\ar[d]&K\{x\}\ar[r]\ar[d]&K[\![x]\!]\\
\Gamma[p^{-1}]\ar@{=}[d]&\Gamma_{\con}[p^{-1}]\ar[l]\ar[r]\ar@{=}[d]&\Fil_{\lambda_1}\Gamma_{\an,\con}\ar[r]&\Fil_{\lambda_2}\Gamma_{\an,\con}\ar[r]&\Gamma_{\an,\con}\ar@{=}[d]&\\
\mathcal{E}&\mathcal{E}^{\dagger}\ar[l]\ar[rrr]&&&\mathcal{R}.&
}\]

\subsection*{Acknowledgement}

The author is indebted to Professor Nobuo Tsuzuki for drawing the author's attention to Nakagawa's paper \cite{Nak}, and suggesting Remark~\ref{rem:PBQ}. The author thanks the referees for detailed comments. Funding The author is supported by Research Fellowships of Japan Society for the Promotion of Science for Young Scientists.


\begin{thebibliography}{}
\bibitem[And02]{And0}
Y.~Andr\'e, Filtrations de type Hasse-Arf et monodromie $p$-adique, Invent. Math.  148  (2002),  no. 2, 285--317.
\bibitem[And08]{And}
Y.~Andr\'e, Dwork's conjecture on the logarithmic growth of solutions of $p$-adic differential equations.  Compos. Math.  144  (2008),  no. 2, 484--494.
\bibitem[Ber02]{Ber}
L.~Berger, Repr\'esentations $p$-adiques et \'equations diff\'erentielles, Invent. Math.  148  (2002),  no. 2, 219--284.
\bibitem[CT09]{CT}
B.~Chiarellotto and N.~Tsuzuki, Logarithmic growth and Frobenius filtrations for solutions of $p$-adic differential equations, J. Inst. Math. Jussieu  8  (2009),  no. 3, 465--505.
\bibitem[CT11]{CT2}
B.~Chiarellotto and N.~Tsuzuki, Log-growth filtration and Frobenius slope filtration of $F$-isocrystals at the generic and special points, Doc. Math.  16  (2011), 33--69.
\bibitem[Chr83]{Chr}
G.~Christol, Modules diff\'erentiels et \'equations diff\'erentielles $p$-adiques, Queen's Papers in Pure and Applied Mathematics, 66. Queen's University, Kingston, ON,  1983. vi+218 pp.
\bibitem[Dwo73a]{Dwo}
B.~Dwork,  On $p$-adic differential equations. II. The $p$-adic asymptotic behavior of solutions of ordinary linear differential equations with rational function coefficients, Ann. of Math. (2)  98  (1973), 366--376.
\bibitem[Dwo73b]{Dwob}
B.~Dwork,  On $p$-adic differential equations. III. On $p$-adically bounded solutions of ordinary linear differential equations with rational function coefficients, Invent. Math. 20 (1973), 35--45.
\bibitem[Dwo82]{book}
B.~Dwork,  Lectures on $p$-adic differential equations, With an appendix by Alan Adolphson, Grundlehren der Mathematischen Wissenschaften, 253, Springer-Verlag, New York-Berlin,  1982. viii+310 pp.
\bibitem[Ked04]{Ann}
K.~Kedlaya, A $p$-adic local monodromy theorem, Ann. of Math. (2)  160  (2004),  no. 1, 93--184.
\bibitem[Ked05]{Doc}
K.~Kedlaya, Slope filtrations revisited, Doc. Math.  10  (2005), 447--525.
\bibitem[Ked10]{pde}
K.~Kedlaya, $p$-adic differential equations, Cambridge Studies in Advanced Mathematics, 125. Cambridge University Press, Cambridge,  2010. xviii+380 pp.
\bibitem[Lut37]{Lut}
E.~Lutz, Sur l'\'equation $y^2=x^3-Ax-B$ dans les corps $p$-adiques, J. Reine Angew. Math. 177 (1937), 238--247.
\bibitem[Meb02]{Meb}
Z.~Mebkhout,  Analogue $p$-adique du th\'eor\`eme de Turrittin et le th\'eor\`eme de la monodromie $p$-adique, Invent. Math.  148  (2002),  no. 2, 319--351.
\bibitem[Nak13]{Nak}
T.~Nakagawa, The logarithmic growth of an element of the Robba ring which satisfies a Frobenius equation, Tohoku Math. J. (2)  65  (2013),  no. 2, 179--198.
\end{thebibliography}
\end{document}